\documentclass[12pt]{amsart}
\usepackage{fullpage}
\usepackage{hyperref}
\usepackage{enumerate}
\usepackage{graphicx}
\usepackage{tikz}
\usepackage{verbatim}

\newcounter{Theorem}

\numberwithin{equation}{section}
\numberwithin{Theorem}{section}

\newtheorem{theorem}[Theorem]{Theorem}
\newtheorem{corollary}[Theorem]{Corollary}
\newtheorem{lemma}[Theorem]{Lemma}

\theoremstyle{definition}
\newtheorem{definition}[Theorem]{Definition}

\newtheorem{example}[Theorem]{Example}
\theoremstyle{remark}
\newtheorem{remark}[Theorem]{Remark}
\newcommand{\C}{\mathbb{C}}
\newcommand{\N}{\mathbb{N}}
\newcommand{\Q}{\mathbb{Q}}
\newcommand{\Qn}{{\mathbb{Q}^n}}
\newcommand{\R}{\mathbb{R}}
\newcommand{\Rn}{{\mathbb R^n}}
\newcommand{\T}{\mathbb{T}}

\newcommand{\Z}{\mathbb{Z}}
\newcommand{\Zn}{{\mathbb Z^n}}

\newcommand{\dist}{\operatorname{dist}}
\newcommand{\supp}{\operatorname{supp}}
\newcommand{\spa}{\operatorname{span}}
\newcommand{\tr}{\operatorname{tr}}
\newcommand{\ov}{\overline}
\newcommand{\ch}{\mathbf 1}
\newcommand{\al}{\alpha}
\newcommand{\ga}{\gamma}
\newcommand{\Ga}{\Gamma}
\newcommand{\ve}{\varepsilon}
\newcommand{\vp}{\varphi}
\newcommand{\la}{\lambda}
\newcommand{\lan}{\langle}
\newcommand{\ran}{\rangle}

\newcommand{\vol}{\operatorname{vol}}

\newcommand{\bq}{\begin{equation}}
\newcommand{\eq}{\end{equation}}

\begin{document}

\title{Meyer wavelets for rational dilations}

\author{Marcin Bownik}

\address{Department of Mathematics, University of Oregon, Eugene, OR 97403--1222, USA}

\email{mbownik@uoregon.edu}


\thanks{The  author was partially supported by NSF grant DMS-2349756. The author wishes to thank Alexander Polishchuk for suggesting an elegant proof of Lemma \ref{cayley}.}

\date{\today}

\begin{abstract}
 We show the existence of smooth band-limited multiresolution analysis (MRA) for any expansive dilation with real entries in any spatial dimension.
We then prove the existence of orthonormal Meyer wavelets, which have smooth and compactly supported Fourier transform, for any expansive dilation with rational entries and any spatial dimension. This extends one dimensional results of Auscher \cite{Au0, Au1}. In a converse direction, we show that well-localized orthogonal MRA wavelets, such as Meyer wavelets, can only exist for expansive dilations with rational entries. This shows the optimality of our existence result and extends one dimensional result of Lemari\'e-Rieusset  \cite{KLR}.
\end{abstract}

\maketitle

% LINE SPACING
%\setlength{\baselineskip}{16pt}

\section{Introduction}\label{S1}

An orthonormal wavelet is a function $\psi$ on $\R$ such that the collection
$\{2^{j/2} \psi(2^jx-k): j\in\Z, k\in\Z\}$ is an orthonormal basis of $L^2(\R)$.
Two most prominent construction of orthonormal wavelets on the real line are Daubechies smooth compactly supported wavelets \cite{Dau0, Dau} and Meyer wavelets with compactly supported $C^\infty$ Fourier transform \cite{Mey0, Mey}. These constructions were preceded by the discovery of the Haar function $\psi = \ch_{[0,1/2]} - \ch_{[1/2,1]}$, which is the earliest example of an orthonormal wavelet \cite{Haa}, and Str\"omberg's spline wavelet with exponential decay \cite{St}, which was shown to form an unconditional basis for Hardy spaces $H^p(\R)$.

%%% from Proposal

The majority of research on higher dimensional wavelets was initially concentrated on isotropic wavelets and the classical function spaces associated with dyadic dilation structure, see \cite{Mey}. By an isotropic wavelet we mean a collection
$\Psi=\{\psi^1,\ldots,\psi^L \} \subset L^2(\R^n)$, (most of the time $L=2^n-1$), such that 
$\{2^{jn/2} \psi^l(2^jx-k): j\in\Z, k\in\Z^n, l=1,\ldots,L\}$ is an orthonormal basis of $L^2(\R^n)$. 
The isotropic wavelets with desired properties can be obtained from the one
dimensional wavelets with the corresponding properties by the tensoring technique due to Gr\"ochenig \cite{Gro0}, see also \cite{Dau, Mey, Woj}. Therefore, the isotropic theory in several dimensions does not cause new
difficulties beyond those appearing in one dimension and is very well understood. The construction of wavelets is based on the concept of a multiresolution analysis (MRA), which was introduced by Mallat \cite{Mall}.

A truly multidimensional wavelet theory takes into account other
possible dilation structures given by the action of an invertible matrix $A\in GL_n(\R)$. 
In this setting an orthonormal wavelet is defined as a collection 
$\Psi=\{\psi^1,\ldots,\psi^L \} \subset L^2(\R^n)$ such that the wavelet (or affine)
system
\begin{equation}\label{affine}
\{ |\det A|^{j/2} \psi^l(A^jx-k): j\in\Z, \ k\in \Gamma, \ 
l=1,\ldots,L \}
\end{equation}
forms an orthonormal basis for $L^2(\R^n)$. Here, $\Gamma$ is a full rank lattice in $\R^n$ that is often chosen to be $\Ga=\Z^n$. Unless the dilation $A$ is very
special, e.g., $A=2\mathbf I$, there is no general procedure of constructing such
wavelets from one dimensional ones.

A significant progress in this direction has been made 
for the class of expansive dilations with integer entries. Recall that a dilation $A$ is expansive if all of its eigenvalues $\la$ satisfy $|\la|>1$.
Gr\"ochenig and Madych \cite{GM}, 
and Lagarias and Wang \cite{LW0, LW3, LW4, LW5} studied the existence of Haar type wavelets 
for integer dilations $A$, which turned out to be a challenging problem with connections in algebraic number theory \cite{LW1, LW2}. The problem of extending Daubechies' construction \cite{Dau} of smooth compactly supported wavelets to higher dimensions was considered by Ayache \cite{Ay1, Ay2}, Belogay, Wang \cite{BW}, and more recently in \cite{FHT}.

A natural class of well-localized wavelets are $r$-regular wavelets introduced by Meyer \cite{Mey}. A function $\psi$ is $r$-regular, where $r=0,1,2, \ldots$, or $\infty$, if $\psi$ is $C^r$ and it has polynomially decaying partial derivatives of order $\le r$. Regular wavelets are an excellent tool to study various function spaces. 
Meyer \cite{Mey} showed that $r$-regular wavelets associated with the dyadic dilation $A=2\mathbf I$ form an unconditional basis for many function spaces, e.g.~Hardy, H\"older, Sobolev, Besov, etc. Similar results were shown by the author for anisotropic function spaces \cite{B1, Bo20, Bo30}.
Therefore, it is important to construct wavelets with nice smoothness and decay properties. 
For any integer dilation $A$, which supports a self-similar tiling of $\R^n$, 
Strichartz \cite{Str} constructed $r$-regular wavelets for all $r\in\N$. However, 
there are examples in $\R^4$ of dilation matrices without self-similar 
tiling \cite{LW1,LW2}. In \cite{B2} the author showed that 
for every integer dilation and $r\in\N$ 
there is an $r$-regular wavelet basis with 
an associated $r$-regular multiresolution analysis. In addition, these 
wavelets can be constructed with all vanishing moments. In a similar vein, Han \cite{Han} showed the existence of $C^r$ wavelets with exponential decay. Finally, Speegle and the author \cite{BS}  
showed that for all integer dilations in two dimensions there exists wavelets  $\infty$-regular Meyer wavelets. Meyer wavelets have $C^\infty$ compactly supported Fourier transform. 

\begin{theorem}\label{intro0}
 For every expansive $2\times 2$ integer dilation $A$,
there exists a wavelet consisting of
$(|\det A|-1)$ band-limited Schwartz class functions.
\end{theorem}

Beyond the class of integer dilations, non-existence results for well-localized wavelets associated were obtained by Chui, Shi \cite{CS} and the author \cite{Bo8} in one dimension, and in higher dimensions \cite{Bo10}. Auscher \cite{Au0} proved that there exist Meyer wavelets, which are smooth and compactly supported in the Fourier domain, for every rational dilation factor. Unlike the dyadic case, the number of generators $L$ of a wavelet associated with the rational dilation factor $p/q$, where $p>q \in \Z$ are relatively prime, needs to be a multiple of $p-q$. The author \cite{Bo10} showed the non-existence of well-localized wavelets, including Meyer wavelets, for irrational dilation factors. However, the question of the existence of Meyer wavelets  for non-integer expansive dilations in dimensions $\ge 2$, and the same problem for integer dilations in dimensions $\ge 3$, remained open. Beyond the setting of $\R^n$, Auscher and Hyt\"onen \cite{AuHy, AuHy2} constructed H\"older regular orthonormal wavelet bases in a general space of homogeneous type.

In contrast, much more is known about the existence of frame wavelets for expansive dilations. We say that $\Psi=\{\psi^1,\ldots,\psi^L \} \subset L^2(\R^n)$ is a Parseval wavelet if the corresponding wavelet system \eqref{affine} is a Parseval frame in $L^2(\R^n)$. Gr\"ochenig and Ron \cite{GR}, based on Ron and Shen's method \cite{RS2}, constructed Parseval wavelet frames for any integer dilation matrix, which are generated by compactly supported
functions with arbitrarily high smoothness. A general construction procedure of wavelet frames via multiresolution analysis (MRA) was shown by Daubechies, Han, Ron, and Shen \cite{DHRS}. Thus, construction of smooth compactly supported Parseval frame wavelets (for the class of integer dilations) does not cause difficulties. Likewise, for every expansive dilation $A$ with real entries, one can easily show the existence of a non-orthogonal singly generated Parseval wavelet $\psi$ in the Schwartz class such that $\widehat \psi \in C^\infty$ is compactly supported, see \cite{B1}. In general, the construction of orthonormal wavelets is a more challenging than that of Parseval frame wavelets.

%The goal of this paper is to characterize all expansive dilations $A$ for which there exists Meyer wavelets associated with a multiresolution analysis.

The goal of this paper is to construct orthonormal Meyer wavelets for any expansive dilation $A$ with rational entries and in any spatial dimension. We also show the optimality of our construction by proving that a well-localized  orthonormal wavelet, which is associated with an MRA, can only exist for dilations $A$ with rational entries. There are three main results of the paper. First we show the existence of smooth band-limited MRAs for any expansive dilation with real entries in any spatial dimension.

\begin{theorem}\label{intro1} Let $A$ be a real $n\times n$ expansive dilation. Then, there exists an MRA $\{V_j\}_{j\in\Z}$ of multiplicity $N\in \N$ such that its scaling functions $\{\vp^1, \ldots, \vp^N\} \subset V_0$ are in the Schwartz class. Moreover, each $\widehat \vp_i$ is a $C^\infty$ real-valued function with compact support.
\end{theorem}

Theorem \ref{intro1} appears to be new already in one dimension. On the surface, Theorem \ref{intro1} contradicts a result of Auscher \cite[Theorem 1]{Au1}, which limits the existence of MRAs on the real line with a well-localized scaling function to rational dilation factors. However, the result in \cite{Au1} is merely incorrectly stated. The correct formulation can be found in Auscher's Ph.D. thesis \cite[Proposition 1.3]{Au0}, which does not lead to any contradiction. 
%This misapprehension might have be the reason which inhibited an earlier discovery of Theorem \ref{intro1}.
Using Theorem \ref{intro1} we show the existence of Meyer wavelets for expansive dilations with rational entries. Theorem \ref{intro3} is an extension of Auscher's result \cite{Au0, Au1} who proved the existence of rationally dilated Meyer wavelets in one dimension.

\begin{theorem}\label{intro2} Let $A$ be a rational $n\times n$ expansive dilation. Then, for some $L\in \N$, there exists an orthonormal wavelet $\Psi=\{\psi^1,\ldots, \psi^{L}\}$ in the Schwartz class, which is associated with an MRA as in Theorem \ref{intro1}. Moreover, each $\widehat \psi^i$ is a $C^\infty$ function with compact support.
\end{theorem}

Finally, we show the optimality of Theorem \ref{intro2}. That is, we prove that well-localized wavelets, such as Meyer wavelets, which are associated with an MRA, can only exist for dilations with rational entries. Theorem \ref{intro3} is an extension of the result of Lemari\'e-Rieusset  \cite[Chapter 3]{KLR} who proved the non-existence of well-localized MRA wavelets for irrational dilation factors on the real line.

\begin{theorem}\label{intro3}
Let $A$ be $n\times n$ real expansive matrix. Suppose that there exists %$(A,\Z^n)$-
an orthonormal wavelet $\{\psi^1,\ldots, \psi^L\} \subset L^2(\R^n)$ such that  each $\widehat \psi^i$ is a $C^\infty$ function with compact support and the space of negative dilates 
\begin{equation*}
V_0= \ov{\spa}\{ D_{A^j} T_k \psi^l: j<0, k\in \Z^n, l=1,\ldots,L\}
\end{equation*}
is shift invariant. Then, $A$ is a rational matrix.
\end{theorem}

In Section \ref{S2} we construct a shift invariant (SI) basis of $L^2(\R^n)$, which is reminiscent of (though distinct from) the Wilson system \cite{DJJ}. The generators of our basis have smooth compactly supported Fourier transforms whose supports produce a symmetric covering of the frequency space. The construction is also related to the local Fourier basis, which was introduced by Coifman and Meyer \cite{CM2} and studied systematically by Auscher, Weiss, and Wickerhauser \cite{AWW}. In the analogy with local Fourier basis, generators of our SI basis  can have arbitrarily small overlaps of frequency supports. 

We prove Theorem \ref{intro1} in Section \ref{S3}. Using Cayley's transformation we first show the existence of an ellipsoid $\mathcal E$ in $\R^n$, which is expansive with respect to the action of the dilation $A$, such that principal axes of $\mathcal E$ have rational coordinates. An appropriate cover of $\mathcal E$ by supports of generators of our SI system, see Figure 2, produces an MRA of possibly high multiplicity which is associated to some rational lattice $\Gamma \subset \Q^n$  of translates. This is then use to show the existence of an MRA of possibly even higher multiplicity which is associated with the standard lattice $\Z^n$ of translates. Then, a standard scaling argument shows the existence of an MRA associated with an arbitrary lattice of translates.

In Section \ref{S4} we restrict our attention to rational dilations in order to develop construction procedure of wavelets out of an MRA. Such constructions are well understood for dilations with integer dilations \cite{B2, CHM, Han, KLR, Woj}. To extend the MRA wavelet construction procedure to rational dilations, it becomes compulsory to assume extra shift-invariance on the core space $V_0$ of an MRA by the result of Hoover and the author \cite{BH}. Namely, we need to enlarge the standard lattice $\Z^n$ to include all translates in $\Gamma=A\Z^n +\Z^n$. The existence of such MRA in the Schwartz class is guaranteed by Theorem \ref{intro1}. Once this framework is adopted, the construction of rationally dilated wavelets follows a similar scheme as that for integer dilations. The refinability of the scaling vector of an MRA implies the existence of a matrix-valued low-pass filter. The corresponding wavelets are then constructed by
appropriate choice of a matrix-valued low-pass filter. This leads to the matrix completion problem, which asks for extending prescribed orthogonal row functions to a square matrix function which is unitary. While measurable matrix extension always works, there are topological obstructions for the existence of continuous matrix extensions, see \cite{CMX, PR}. Despite this, by exploiting the existence of an MRA of arbitrary high multiplicity we apply the matrix completion result of Ashino and Kametani \cite{AK} to deduce Theorem \ref{intro2} in Section \ref{S5}.

In Section \ref{S6} we generalize the results of Speegle and the author \cite{BS} on the existence  MRAs and Meyer wavelets associated with strictly expansive dilations with integer entries to the setting of non-integer dilations. We show that MRAs of multiplicity 1 exist for  strictly expansive dilations with real entries. This implies the existence of Meyer wavelets for rational dilations $A$, which are strictly expansive, with the smallest possible number of generators. In light of Theorem \ref{intro0} the problem of characterizing dilations $A$ for which there exist Meyer wavelets with a minimal number of generators remains open in dimensions $\ge 3$ for integer dilations.

Finally, in Section \ref{S7} we show the optimality of Theorem \ref{intro2}. Meyer wavelets associated with an MRA can only exist for dilations with rational entries. This is done by extending techniques on shift-invariant operators with localized kernels, which were developed by Lemari\'e-Rieusset \cite{KLR, Lem2}. This involves combining several results such as: a crucial lemma on anisotropic Sobolev spaces associated with Beurling weights due to Coifman and Meyer \cite{CM1}, an extension of the extra invariance lemma of Lemari\'e-Rieusset \cite{KLR} to higher dimensions, results on spaces with extra shift invariance \cite{ACP}, and properties of the spectral function of SI spaces \cite{BR1, BR2}.

\section{
Symmetric frequency supported shift invariant basis in $L^2(\Rn)$
}\label{S2}

In this section we construct a shift invariant (SI) basis of $L^2(\R^n)$ that plays a key role in the subsequent construction of Meyer wavelets for rational dilations. The generators of our SI basis have symmetric supports in frequency, which is reminiscent of the Wilson system \cite{DJJ}. The novel feature of our construction is that SI generators have frequency supports that are arbitrarily close to minimal, which correspond to the symmetric partition of the frequency domain $\R$ into sets 
\begin{equation}\label{sypa}
[-j/2-1/2,-j/2] \cup [j/2, j/2+1/2], \qquad j\ge 0.
\end{equation}

%minimal frequency supports. Hence, it is optimal with respect to minimality of frequency support

The famous Balian-Low Theorem \cite{Gro, HW} states that a window $g\in L^2(\R)$ of Gabor system  $\{T_k M_m g \}_{k,m\in \Z}$, which forms an orthonormal basis of $L^2(\R)$, can not have simultaneously good time and frequency (Fourier transform) localization. Here, the translation operator $T_{k}$ by $k\in \R$ and the modulation operator $M_{m}$ by $m\in \R$ are defined as
\[
 T_{k} f(x) = f(x-k), \ \ M_{m} f(x) = e^{2\pi i x  m} f(x), \ \ f\in L^{2}(\R), \, x \in \R.
\]
The Wilson basis is a clever modification of the Gabor system that overcomes this limitation by taking simple linear combinations of time-frequency shifts of a window function. The Wilson system with a window $g \in L^2(\R)$ is defined as a shift-invariant system with generators
\begin{equation}\label{wilson0}
 g, \quad   \{\tfrac{1}{\sqrt{2}} (M_{m}+(-1)^m M_{-m})g \}_{m\in \N}, \quad
\{\tfrac{1}{\sqrt{2}} T_{1/2} (M_{m}-(-1)^m M_{-m})g \}_{m\in \N}.
\end{equation}
If the window function $g$ is sufficiently localized in frequency, then apart from the pure translations $\{T_{k} g\}_{k\in\Z}$, the Wilson system produces a
symmetric covering of the frequency line, in the sense that each element of
the system has two peaks around $\xi=\pm m$ in its spectrum.

Daubechies, Jaffard, and Journ\'e \cite{DJJ} constructed a Wilson orthonormal basis consisting of functions with exponential decay both in time and frequency. Auscher \cite{Au2} showed that the Wilson basis can be placed in a general framework of local Fourier bases. A  comprehensive exposition of local Fourier bases using smooth orthogonal projections can be found in the book of Hern\'andez and Weiss \cite[Chapter 1]{HW}. These include local sine and cosine bases of Coifman and Meyer \cite{CM2}, which were also studied by Malvar \cite{Mal} in connection to signal processing, and then used to construct smooth wavelets by Auscher, Weiss, and Wickerhauser \cite{AWW}. Laeng \cite{Lae} gave a closely related construction of orthonormal basis of $L^2(\R)$ whose elements have good time-frequency localization and supports adapted to any symmetric partition of frequency space. Wilson bases in higher dimensions were studied in \cite{BJLO}. The construction of symmetric frequency supported basis is motivated by and has some common elements with these developments.

We start by defining a symmetric frequency supported SI system in $L^2(\R)$.

\begin{definition}\label{wilson}
Let $\phi:\R \to \C$ be a $C^\infty$ function such that:
\begin{enumerate}[(i)]
%\item $\phi$ is even, $\phi(\xi)=\phi(-\xi)$ for all $\xi\in\R$,
\item $\supp \phi \subset [-1/4-\delta,1/4+\delta]$ for some $0<\delta\le 1/4$,
\item $\sum_{k\in\Z} |\phi(\xi+k/2)|^2=1$ for all $\xi\in\R$,
\item $\phi^{(m)}(0)=0$ for all $m\ge 1$.
\end{enumerate}

Define the set of generators $\mathcal W(\phi)=\{f_j: j=0,1, \ldots \}$ by
\begin{align}
\label{fj0}
\widehat f_0(\xi) &= \begin{cases}
\phi(\xi-1/4) & \xi > 1/4, \\
\phi(0) & |\xi| \le 1/4, \\
\phi(\xi+1/4) & \xi < -1/4,
\end{cases}
\\
\label{fj}
\widehat f_j(\xi) &= \ov{\phi(\xi- j/2-1/4)}^{(j)} + (-1)^j \ov{\phi(\xi+ j/2+1/4)}^{(j)} \qquad\text{for }j \ge 1.
\end{align}
Here, the Fourier transform of $f$ is given by $\hat f(\xi) = \int_\R f(x) e^{-2\pi i x \xi} dx$ and we use the following convention for taking complex conjugates 
\[
\ov{z}^{(j)}=\begin{cases} \ov{z} & j \text{ is odd}\\
z & j \text{ is even.}
\end{cases}
\]
The symmetric frequency supported (SFS) shift-invariant  system generated by $\mathcal W(\phi)$ is given by
\[
E^\Z(\mathcal W(\phi)):= \{T_k f_j: k\in \Z, \ j=0,1,\ldots \}.
\]
\end{definition}

\begin{figure}\label{figure1}
\centerline{\includegraphics[width=5in]{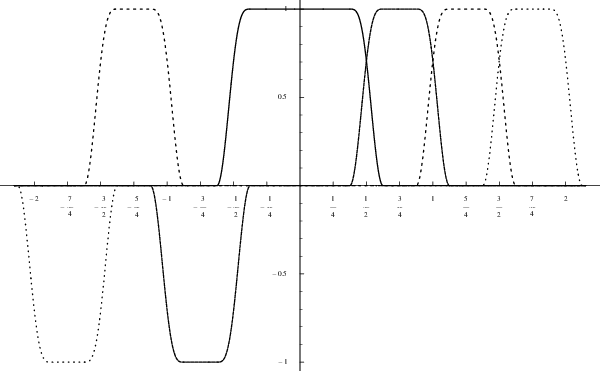}}
\caption{Graphs of functions $\widehat f_j$, $j=0,1,2,3$.}
\end{figure}

\begin{remark}
The condition (iii) in Definition \ref{wilson} is imposed to ensure that $\widehat f_0$ is a $C^\infty$ function. Indeed, $\widehat f_0$ and all its derivatives are continuous at $\xi=\pm 1/4$. Elsewhere, $\widehat f_0$ is automatically smooth. 
\end{remark}

\begin{remark}
Assume the function $\phi$ in Definition \ref{wilson} is real valued.  In this case, the condition (iii) is an immediate consequence of (i) and (ii).  Indeed, (i) and (ii) imply that for all $\xi \in (-1/2,1/2)$,
\[
\phi(\xi-1/2)^2+ \phi(\xi)^2+ \phi(\xi+1/2)^2=1.
\]
Thus, $\phi(0)=\pm 1$.
Since  $\phi^2$ and all of its derivatives vanish at $\xi=\pm 1/2$, all derivatives of $\phi^2$ must also vanish at $\xi=0$, thus showing (iii). Consequently, $\widehat f_0$ is automatically a $C^\infty$ function.
\end{remark}

If $\phi$ is real valued, then so is $\widehat f_0$.  It is not difficult to show that the Gabor system 
\begin{equation}\label{gab}
\{M_m T_{k/2} f_0\}_{ m,k\in \Z},
\end{equation}
generated by a window function $f_0$ with time-frequency lattice $(\Z/2) \times \Z$, is a tight frame frame for $L^2(\R)$. Clearly, the Gabor system \eqref{gab} is not an orthonormal basis since the support of $\widehat f_0$ is contained in the interval $[-1/2-\delta,1/2+\delta]$ of length $<3/2$. However, by the theorem of Daubechies, Jaffard, and Journ\'e \cite{DJJ},  the Wilson system with window $f_0$, which is given by
\begin{align*}
 \{T_{k}f_0\}_{k\in\Z} & \cup \{\tfrac{1}{\sqrt{2}} T_{k} (M_{m}+(-1)^m M_{-m})f_0\}_{k\in\Z,m\in \N} \\
& \cup \{\tfrac{1}{\sqrt{2}} T_{k}T_{1/2} (M_{m}-(-1)^m M_{-m})f_0\}_{k\in\Z,m\in \N},
\end{align*}
is an orthonormal basis of $L^2(\R)$.

In general, two generators of the Wilson system with window $g$
\[
 \tfrac{1}{\sqrt{2}} (M_{m}+(-1)^m M_{-m})g  \quad\text{and}\quad
\tfrac{1}{\sqrt{2}} T_{1/2} (M_{m}-(-1)^m M_{-m})g, \qquad m\in \N
\]
have frequency supports with large overlaps.
This is due to the fact that the above generators have contributions from both $M_{m}g$ and $M_{-m}g$. In contrast, our construction of the symmetric frequency supported (SFS) system  does not suffer from such large overlaps. Indeed, SI generators $f_j$ can have arbitrarily small frequency overlaps as $\delta \to 0$ approximating the partition \eqref{sypa}. At the same time, the SFS system retains the most important feature of the Wilson system by forming an orthonormal basis of $L^2(\R)$.

\begin{theorem}\label{wil}
The shift invariant system generated by $\mathcal W(\phi)=\{f_j: j=0,1, \ldots \}$
\[
E^\Z(\mathcal W(\phi))= \{T_k f_j: k\in \Z, \ j=0,1,\ldots \}
\]
is an orthonormal basis of $L^2(\R)$.
\end{theorem}

\begin{proof}
Observe that
\[
||\phi||^2 = \int_\R |\phi(\xi)|^2 d\xi = \sum_{k\in\Z} \int_0^{1/2} |\phi(\xi+k/2)|^2 = 1/2.\]
Therefore, a simple calculation shows that
\[
||f_0||^2 = ||\phi||^2 + 1/2 =1.
\]
Likewise, since the supports of two components defining $\widehat f_j$ are disjoint we have
\[
||f_j||^2 = 2 ||\phi||^2 = 1 \qquad\text{for }j \ge 1.
\]
Hence, to establish Theorem \ref{wil}, by \cite[Proposition 3.2.1]{Dau} it suffices to prove that $E(\mathcal W(\phi))$ is a Parseval frame, i.e.,
\bq\label{w1}
\sum_{j=0}^\infty \sum_{k\in\Z} | \lan h, T_k f_j \ran|^2 = ||h||^2 \qquad\text{for all }h\in L^2(\R).
\eq

Let $\mathcal T: L^2(\R) \to L^2(\T,\ell^2(\Z))$ be a fiberization operator given by $\mathcal T f(\xi)= ( \widehat f(\xi+k))_{k\in \Z}$, where $\T=\R/\Z$ is identified with $[-1/2,1/2)$. By \cite[Theorem 2.3]{B0}, a frame or Riesz basis property of a SI system is characterized by the corresponding property in terms of fibers. In particular,
$E(\mathcal W(\phi))$ is Parseval frame of $L^2(\R)$ if and only if a collection $(\mathcal T f_j(\xi) )_{j\in \N_0}$ is a  Parseval frame of $\ell^2(\Z)$ for a.e. $\xi\in \T$. In the terminology of Ron and Shen \cite{RS}, this can be equivalently stated that the dual Gramian of $E(\mathcal W(\phi))$ is an identity on $\ell^2(\Z)$ for a.e. $\xi$. That is, we need to show that
\bq\label{w3}
 \sum_{j=0}^\infty \widehat f_j(\xi) \ov{\widehat f_j(\xi+k)} = \delta_{k,0} \qquad\text{for a.e. }\xi\in\R \text{ and } k\in\Z.
\eq 
Alternatively, \eqref{w1} follows from \eqref{w3} by a direct calculation using the Plancherel theorem and a periodization argument
\[\label{w2}
\begin{aligned}
\sum_{j=0}^\infty \sum_{k\in\Z} | \lan h, T_k f_j \ran|^2  & = 
\sum_{j=0}^\infty \sum_{k\in\Z} | \lan \widehat h, M_k \widehat f_j \ran|^2 
\\
&=
\sum_{j=0}^\infty \sum_{k\in\Z} \bigg|  \int_{[0,1]} e^{-2\pi i \xi k} \bigg(\sum_{m\in \Z} \widehat h(\xi+m)\ov{\widehat f_j(\xi+m)}\bigg) d\xi\bigg|^2 
\\
& = \sum_{j=0}^\infty  \int_{[0,1]} \bigg| \sum_{m\in \Z} \widehat h(\xi+m)\ov{\widehat f_j(\xi+m)}\bigg|^2 d\xi
\\
& =   \int_{[0,1]} \sum_{j=0}^\infty \sum_{m\in \Z} \sum_{k\in \Z}  \widehat h(\xi+m) \ov{\widehat f_j(\xi+m)}  \widehat f_j(\xi+m+k)  \ov{\widehat h(\xi+m+k)}  d\xi
\\
&= \sum_{k\in\Z}
\int_{\R} \widehat h(\xi) \ov{\widehat h(\xi+k)} \bigg( \sum_{j=0}^\infty \ov{\widehat f_j(\xi)}  \widehat f_j(\xi+k) \bigg)  d\xi.
\end{aligned}
\]
To justify the above manipulations we need to assume that $h\in L^2(\R)$ is such that $\widehat h$ has bounded support.
Combining the above with \eqref{w3} yields \eqref{w1} for functions $h$ belonging to a dense subset of $L^2(\R)$. By the density, \eqref{w1} holds for all $h\in L^2(\R)$.

To establish \eqref{w3}, note that by support considerations we have 
\bq\label{sp}
\supp \widehat f_j \subset [-j/2-3/4,-j/2+1/4] \cup [j/2-1/4, j/2+3/4]
\qquad\text{for }j \ge 0.
\eq
Thus, $\widehat f_j(\xi) \ov{\widehat f_j(\xi+k)} \equiv 0$, unless $k=0$, $|k|=j$, or $|k|= j+1$. In the special case of $k=0$, we have
\[
\begin{aligned} 
|\widehat f_j(\xi)|^2& =
\begin{cases}
|\phi(\xi-1/4)|^2 & \xi > 1/4, \\
1 & |\xi| \le 1/4, \\
|\phi(\xi+1/4)|^2 & \xi < -1/4.
\end{cases}
\qquad\text{for }j=0,\\
|\widehat f_j(\xi)|^2& = |\phi(\xi- j/2-1/4)|^2 + |\phi(\xi+ j/2+1/4)|^2 \qquad\text{for }j\ge 1.
\end{aligned}
\]
Hence, by (ii) we have that for $\xi>1/4$
\begin{equation}\label{fj1}
\sum_{j=0}^\infty |\widehat f_j(\xi)|^2 = \sum_{j=0}^\infty |\phi(\xi- j/2-1/4)|^2 = \sum_{j\in\Z} |\phi(\xi - j/2-1/4)|^2 =
1.
\end{equation}
A similar identity holds for $\xi<-1/4$. Finally, if $|\xi| \le 1/4$, then $f_j(\xi)=0$ for all $j\ge 1$ and we have
\[
\sum_{j=0}^\infty |\widehat f_j(\xi)|^2 =1 \qquad\text{for all }\xi\in\R.
\]
In the case when $k\ne 0$, a direct calculation using \eqref{fj} and the support condition \eqref{sp} shows that for $j\ge 1$ and $k=\pm j$,
\begin{equation}\label{fj2}
\begin{aligned}
\widehat f_j(\xi) \ov{\widehat f_j(\xi +k)} 
 = &
\left(\ov{\phi(\xi- j/2-1/4)}^{(j)} + (-1)^j \ov{\phi(\xi+ j/2+1/4)}^{(j)} \right)
 \\
&\cdot \left( \ov{\phi(\xi \pm j - j/2-1/4)}^{(j+1)} + (-1)^j \ov{\phi(\xi \pm j + j/2+1/4)}^{(j+1)} \right)
\\ 
= & \begin{cases}  (-1)^j \ov{\phi(\xi+ j/2+1/4)}^{(j)}  \ov{\phi(\xi + j/2 -1/4)}^{(j+1)} & k=j \\
(-1)^j  \ov{\phi(\xi- j/2-1/4)}^{(j)} \ov{\phi(\xi - j/2+1/4)}^{(j+1)} & k=-j
 \end{cases}
\\
= &
(-1)^j \ov{\phi(\xi \pm (j/2 +1/4))}^{(j)} \ov{\phi(\xi \pm (j/2 -1/4))}^{(j+1)}.
\end{aligned}
\end{equation}
Likewise, for $j \ge 1 $ and $k=\pm (j+1)$ we have
\[
\begin{aligned}
\widehat f_j(\xi) \ov{\widehat f_j(\xi +k)}  = &
\left(\ov{\phi(\xi- j/2-1/4)}^{(j)} + (-1)^j \ov{\phi(\xi+ j/2+1/4)}^{(j)} \right)
 \\
&\cdot \left( \ov{\phi(\xi \pm (j+1) - j/2-1/4)}^{(j+1)} + (-1)^j \ov{\phi(\xi \pm (j+1) + j/2+1/4)}^{(j+1)} \right)
\\ 
= & \begin{cases}  (-1)^j \ov{\phi(\xi+ j/2+1/4)}^{(j)}  \ov{\phi(\xi + j/2+3/4)}^{(j+1)} & k=j+1 \\
(-1)^j  \ov{\phi(\xi- j/2-1/4)}^{(j)} \ov{\phi(\xi - j/2-3/4)}^{(j+1)} & k=-j-1
 \end{cases}
\\
= &
(-1)^j 
\ov{\phi(\xi \pm (j/2+1/4))}^{(j)}\ov{\phi(\xi \pm (j/2+3/4))}^{(j+1)}.
\end{aligned}
\]
A similar calculation using \eqref{fj0} yields
\[
\widehat f_0(\xi) \ov{\widehat f_0(\xi \pm 1)} = \phi(\xi \pm 1/4) \ov{ \phi(\xi \pm 3/4)}.
\]
Combining two previous formulas shows that for $j\ge 0$ and $k=\pm (j+1)$ we have
\begin{equation}\label{fj3}
\widehat f_j(\xi) \ov{\widehat f_j(\xi +k)}  = 
(-1)^j 
\ov{\phi(\xi \pm (j/2+1/4))}^{(j)}\ov{\phi(\xi \pm (j/2+3/4))}^{(j+1)}.
\end{equation}
Let $k\ge 1$. Using \eqref{fj2} and \eqref{fj3}, the infinite sum \eqref{w3} has at most two non-zero terms corresponding to $j=k-1$ and $j=k$, which add up to zero,
\begin{multline*}
(-1)^{k-1} \ov{\phi(\xi+(k-1)/2+1/4)}^{(k-1)} \ov{\phi(\xi+(k-1)/2+3/4)}^{(k)}
\\
+ (-1)^k \ov{\phi(\xi + k/2+1/4)}^{(k)} \ov{\phi(\xi + k/2-1/4)}^{(k+1)}=0.
\end{multline*}
Likewise, if $k\le -1$, then the infinite sum \eqref{w3} has at most two non-zero terms corresponding to $j=-k+1$ and $j=-k$ also adds up to zero:
\begin{multline*}
(-1)^{-k+1} \ov{\phi(\xi+(k-1)/2+1/4)}^{(-k+1)} \ov{\phi(\xi+(k-1)/2+3/4)}^{(-k)}
\\
+ (-1)^{-k} \ov{\phi(\xi + k/2+1/4)}^{(-k)} \ov{\phi(\xi + k/2-1/4)}^{(-k+1)}=0.
\end{multline*}
Consequently, by \eqref{fj1} the identity \eqref{w3} holds and Theorem \ref{wil} is proven.
\end{proof}

We can easily extend the SFS basis to higher dimensions by tensoring argument.

\begin{definition}
Define the set of {\it SFS basis} generators in $\R^n$ as
\[
\mathcal W_n(\phi) = \{ w_{\mathbf j}=f_{j_1} \otimes \ldots \otimes f_{j_n}: \mathbf j =(j_1,\ldots,j_n)\in \N_0^n \}.
\]
Here, 
\[
w_{\mathbf j} (x) = (f_{j_1} \otimes \ldots \otimes f_{j_n})(x_1,\ldots x_n) = \prod_{i=1}^n f_{j_i}(x_i) \qquad\text{for } \mathbf j\in\N_0^n, \ x=(x_1,\ldots,x_n)\in\Rn.
\]
\end{definition}

Note that each generator $w_{\mathbf j}$ is in the Schwartz class since $\widehat w_{\mathbf j}$ is a compactly supported $C^\infty$ function. Moreover, by Definition \ref{wilson} we have
\bq\label{supp}
\begin{aligned}
\supp \widehat w_{(j_1,\ldots,j_n)} & \subset I_{j_1} \times \ldots \times I_{j_n}, \quad\text{where}
\\
I_j&  = [-(j+1)/2-\delta,-j/2+\delta] \cup [j/2-\delta,(j+1)/2+\delta].
\end{aligned}
\eq
Thus, a typical SFS basis generator $\widehat w_{\mathbf j}$ is supported within the distance $\delta$ (in $\ell^\infty$ norm) of the union of $2^n$ cubes of the form $\pm [j_1/2,(j_1+1)/2] \times \ldots \times \pm [j_n/2,(j_n+1)/2]$. The only exception is when $j_k =0$ for some $k=1,\ldots,n$. In this case two intervals $\pm [j_k/2,(j_k+1)/2]$ are replaced by a single interval $[-1/2,1/2]$ and consequently $\widehat w_{\mathbf j}$ is supported on a fewer number of cubes. 

As an immediate consequence of Theorem \ref{wil}, we obtain its higher dimensional counterpart. 

\begin{theorem}\label{wiln}
The SI system 
\[
E^\Zn(\mathcal W_n(\phi)):= \{T_{\mathbf k} w_{\mathbf j}: \mathbf k\in \Zn, \  \mathbf j \in \N_0^n\},
\]
generated by $\mathcal W_n(\phi)$ is an orthonormal basis of $L^2(\Rn)$.
\end{theorem}

\begin{proof}
The $n$-fold tensoring of an orthonormal basis of $L^2(\R)$ produces an orthonormal basis in $L^2(\R)^{\otimes n}$, which is identified with $L^2(\Rn)$. 
Since every element of $E^\Z(\mathcal W(\phi))$ is of the form $T_k f_j$, $k\in \Z$, $j\ge 0$, we have
\[
T_{k_1} f_{j_1} \otimes \ldots \otimes T_{k_n} f_{j_n} = T_{(k_1,\ldots,k_n)} w_{(j_1,\ldots,j_n)}.
\]
Therefore, $n$-fold tensoring of orthonormal basis $E^\Z(\mathcal W(\phi))$ of $L^2(\R)$ produces an orthonormal basis $E^\Zn(\mathcal W_n(\phi))$ of $L^2(\R^n)$.
\end{proof}

\section{Schwartz class MRA associated with real dilations}\label{S3}

In this section we prove the existence of smooth MRAs of (possibly) high multiplicity for any expansive real dilations.

\begin{definition}\label{DMRA}
Let $A$ be an $n\times n$ expansive matrix, i.e., all eigenvalues $\lambda$ of $A$ satisfy $|\lambda|>1$. Let $\Gamma$ be a full rank lattice in $\Rn$. We say that a sequence $\{ V_j \}_{j\in \Z}$ of closed subspaces of $L^2(\Rn)$, is a multiresolution analysis (MRA) associated with $(A,\Gamma)$ of multiplicity $N\in \N$, if
\begin{enumerate}[(M1)]
\item $V_j\subset V_{j+1}$  for all $j\in\Z$,
\item $\bigcap_{j\in\Z{}}V_j=\{0\}$,
\item $\overline{\bigcup_{j\in\Z{}}V_j}=L^2(\Rn)$,
\item $f\in V_0 \iff f(A^j \cdot) \in V_j$  for all $j\in\Z$,
\item $\exists \ \Phi=\{\vp^1, \ldots, \vp^N \} \subset V_0$ such that 
$
E^\Gamma(\Phi):=\{T_\gamma \vp: \gamma \in \Gamma, \ \vp \in \Phi\}
$
is an orthonormal basis of $V_0$. 
\end{enumerate}
In short, we say that $\{ V_j \}_{j\in \Z}$ is $(A,\Gamma)$-MRA, and $\Phi$ is a {\it scaling vector}.
\end{definition}

\begin{theorem}\label{MRA} Suppose $A$ is a real $n\times n$ expansive dilation and $\Gamma$ is a full rank lattice. Then, for some $N\in \N$, there exists an $(A,\Gamma)$-MRA $\{V_j\}_{j\in\Z}$ of multiplicity $N$ such that its scaling functions $\{\vp^1, \ldots, \vp^N\} \subset V_0$ are in the Schwartz class. Moreover, each $\widehat \vp_i$ is a $C^\infty$ real-valued function with compact support.
\end{theorem}

Theorem \ref{MRA} appears to be new even in the one dimensional setting, where one can show the existence of an MRA with multiplicity $N=1$. This might look quite suspicious in light of the result by Auscher \cite[Theorem 1]{Au1}, which precludes existence of ``nice'' MRAs with an irrational dilation factor. However, \cite[Theorem 1]{Au1} is incorrectly stated in contrast to Ph.D. thesis of Auscher \cite[Proposition 1.3]{Au0} that contains a correct version of this result.
In the proof of Theorem \ref{MRA} we shall employ a series of lemmas. We start with Lemma \ref{cayley} that is very likely a folklore fact following immediately from the Cayley transform.

\begin{lemma}\label{cayley} The rational orthogonal group $O(n,\Q)$ is a dense subset of real orthogonal group $O(n,\R)$.
\end{lemma}

\begin{proof}
We shall employ Cayley's parametrization of orthogonal matrices. Any orthogonal matrix $U$ that does not have $-1$ as an eigenvalue can be expressed as
\begin{equation}\label{cay1}
U=(\mathbf I+S)(\mathbf I-S)^{-1}
\eq
for some skew-symmetric matrix $S$, $S^{\top}=-S$, where $\mathbf I$ is the identity matrix. In fact, $S$ is given explicitly by $S=(U+\mathbf I)^{-1}(U-\mathbf I)$.

Take any $U\in O(n,\R)$ that does not have $-1$ as an eigenvalue and express it as \eqref{cay1} for some skew-symmetric matrix $S$ with real entries. $S$ can be approximated by skew-symmetric matrices with rational entries. The corresponding orthogonal matrices given by Cayley's parametrization \eqref{cay1} belong to $O(n,\Q)$ and approximate $U$. This is a consequence of the continuity of multiplication and inverse operations of the topological group $GL(n,\R)$. Finally, we observe that the set of orthogonal matrices in $O(n,\R)$ that do not have $-1$ as an eigenvalue is an open and dense subset of $O(n,\R)$.
\end{proof}

The following lemma strengthens a result of Szlenk \cite[Lemma 1.5.1]{Sz} on the existence of expanding ellipsoids, see also \cite[Lemma 2.2]{B1}.

\begin{lemma}\label{sz}  Suppose $B$ is an expansive dilation.  Then, there exists an ellipsoid $\mathcal E$ that is expanding with respect to $B$. That is, there exists some $\la>1$ such that
\bq\label{sz1}
\lambda \mathcal E \subset B( \mathcal E).
\eq
Moreover, $\mathcal E$ can be chosen to have rational coordinates of principal axes. That is,
\bq\label{sz2}
\mathcal E=U^{\top} P U ({\mathbf B}(0,1)),
\eq
where $U\in O(n,\Q)$ and $P$ is a diagonal matrix with positive rational eigenvalues.
\end{lemma}

\begin{proof}
The first part of the statement is Szlenk's lemma which guarantees existence of an ellipsoid $\tilde {\mathcal E}$ satisfying $\tilde \lambda \tilde {\mathcal E} \subset B(\tilde {\mathcal E})$ for some $\tilde \lambda>1$. An ellipsoid $\tilde {\mathcal E}$ can be written as $\tilde {\mathcal E}=H(\mathbf (0,1))$ for some positive-definite hermitian matrix $H$. Thus, $H = \tilde U^{\top} \tilde P \tilde U$ for some $\tilde U \in O(n,\R)$ and $\tilde P$ diagonal matrix with positive diagonal entries. By Lemma \ref{cayley}, $\tilde U$ can be approximated by $U\in O(n,\Q)$. Likewise, $\tilde P$ can be approximated by a diagonal matrix $P$ with rational entries. Hence, for every $\ve>0$ we can find an ellipsoid $\mathcal E$ as in \eqref{sz2} such that
\[
\frac{1}{1+\ve} \mathcal E \subset \tilde{\mathcal E} \subset (1+\ve) \mathcal E.
\]
Thus,
\[
\mathcal E \subset  (1+\ve)\tilde{ \mathcal E} \subset \frac{1+\ve}{\tilde \lambda} B(\tilde{\mathcal E}) \subset \frac{(1+\ve)^2}{\tilde \lambda} B(\mathcal E).
\]
Choosing sufficiently small $\ve>0$ such that $(1+\ve)^2<\tilde \lambda$, shows that \eqref{sz1} holds.
\end{proof}

The following rudimentary lemma enables us to reduce the study of MRAs with a general translation lattice $\Gamma$ to the setting of the standard lattice $\Zn$, albeit for a different dilation matrix, which is similar to the original one. The proof of Lemma \ref{resc} uses a standard change of variables argument. Given a matrix $P\in GL(\R,n)$, define the dilation operator $D_P$ acting on a function $f\in L^2(\R^n)$ by 
\[
D_P f(x) = |\det P|^{1/2} f(Px).
\]

\begin{lemma}\label{resc}
Let $\Gamma$ and $\tilde \Gamma$ be two full rank lattices, and let $P\in GL(\R,n)$ be such that $\Ga=P\tilde \Ga$.
Suppose that $\{V_j\}_{j\in \Z}$ is $(A,\Gamma)$-MRA with a scaling vector $\Phi=\{\vp^1,\ldots,\vp^N\}$. 
Define 
\bq\label{resc0}
\tilde A=P^{-1}AP, \qquad \tilde \Phi=D_P(\Phi)= \{D_P \vp^1,\ldots,D_P \vp^N\}, \qquad \tilde V_j=D_P(V_j),\ j\in \Z.
\eq
Then, $\{\tilde V_j\}_{j\in \Z}$ is $(\tilde A,\tilde \Ga)$-MRA with the scaling vector $\tilde \Phi$. Conversely, any $(\tilde A,\tilde \Ga)$-MRA $\{\tilde V_j\}_{j\in \Z}$ with the scaling vector $\tilde \Phi$ corresponds to $(A,\Gamma)$-MRA $\{V_j\}_{j\in \Z}$ with a scaling vector $\Phi=\{\vp^1,\ldots,\vp^N\}$ via the relation \eqref{resc0}.
\end{lemma}

\begin{proof}
The proof is a simple verification of properties of MRA in Definition \ref{DMRA} using the commutator
relations 
\[ T_k D_P = D_P T_{Pk}  \quad \text{and} \quad D_{(P^{-1}AP)^j}
D_P = D_P D_{A^j}, \ j\in \Z, k\in \Rn. \] 
 It is easy to show that any of the properties (M1)--(M5) of $\{V_j\}_{j\in\Z}$ imply the corresponding property for $\{\tilde V_j\}_{j\in\Z}$. For example, to show (M5) we observe that $D_P$ is an isometric isomorphism of $L^2(\Rn)$ and the image of orthonormal basis $E^\Ga(\Phi)$ under $D_P$ is an orthonormal basis $E^{\tilde \Ga}(\tilde \Phi)=D_P(E^\Ga(\Phi))$ of $\tilde V_0=D_P(V_0)$.
The converse is shown the same way, which completes the proof of the lemma.
\end{proof}

To show the existence of MRAs associated with the standard lattice $\Zn$, we shall employ the following lemma. Unlike the previous lemma, the dilation matrix remains fixed and instead the multiplicity of an MRA changes in Lemma \ref{req}. 

\begin{lemma}\label{req}
Let $A$ be an $n\times n$ expansive matrix with real entries. Suppose that for some rational lattice $\Gamma \subset \Q^n$, there exists $(A,\Ga)$-MRA $\{V_j\}_{j\in\Z}$ of multiplicity $N$. Then, there exists $(A,\Z^n)$-MRA $\{V'_j\}_{j\in\Z}$ of multiplicity $N'$, where $N'$ is a multiple of $N$.
\end{lemma}

\begin{proof}
Since $\Ga$ is a rational lattice, there exists an integer $q$ such that $q\Zn \subset \Ga$. Let $\tilde \Ga= \frac{1}{q}\Ga$. The quotient group $(\tilde \Ga)/\Zn$ has a finite order $r$. Let $d_1,\ldots,d_r$ be representatives of its distinct cosets. Applying Lemma \ref{resc} with $P= q \mathbf I$ we deduce the existence of $(A,\tilde \Ga)$-MRA $\{\tilde V_j\}_{j\in\Z}$ and the scaling vector $\tilde \Phi$ of the same multiplicity $N$ as $\{V_j\}_{j\in\Z}$. Observe that
\[
E^{\tilde \Ga}(\tilde \Phi) = \{ T_{d_i+k} \tilde \vp: \tilde \vp \in \tilde \Phi,\ k\in \Zn, i=1,\ldots, r\} = \bigcup_{i=1}^r E^{\Zn}(T_{d_i}\tilde \Phi).
\]
Thus, the core space $\tilde V_0$ has an orthonormal basis of the form $E^{\Zn}(\Phi')$, where 
\[
\Phi' =  \bigcup_{i=1}^r T_{d_i}\tilde \Phi
\]
consists of $N'=rN$ functions. Consequently, $\{\tilde V_j\}_{j\in\Z}$ can be treated also as $(A,\Zn)$-MRA of multiplicity $N'$ with the scaling vector $\Phi'$. 
This completes the proof of Lemma \ref{req}.
\end{proof}

Analogues of Lemmas \ref{resc} and \ref{req} can be also easily formulated in terms of the existence of orthonormal wavelet bases, frame wavelets, or any other property that is preserved by the change of variables $D_P$ such as orthonormality or frame property. For example, variants of Lemma \ref{resc} have been used, at least implicitly, by many authors. However, the method of Lemma \ref{req} has not been exploited much compared to the standard change of variables argument of Lemma \ref{resc}. We are now ready to give the proof of the main result of this section.

\begin{proof}[Proof of Theorem \ref{MRA}]
By Lemma \ref{resc} it suffices to show the existence of $(A,\Zn)$-MRAs with scaling vector in the Schwartz class for any $n\times n$ real expansive dilation. By Lemma \ref{req}, the standard lattice $\Zn$ can be replaced by a suitable rational lattice $\Gamma \subset \Qn$. Using Lemma \ref{sz}, we shall show how to select $\Ga$, so that the corresponding MRA can be constructed using the SFS basis $E^\Zn(\mathcal W_n(\phi))$ from Theorem \ref{wiln}.

Let $B=A^{\top}$. By Lemma \ref{sz}, there exists an ellipsoid $\mathcal E$ of the form \eqref{sz2}, where $U\in O(n,\Q)$, that is expanding with respect to the action of $B$, i.e., \eqref{sz1} holds. Moreover, by rescaling ellipsoid $\mathcal E$ to $c \mathcal E$, where $c>0$ is sufficiently large, we can assume that all diagonal entries $p_1,\ldots, p_n$ of matrix $P$ satisfy $p_{i} > \sqrt{n}/(\lambda-1)$. Consequently,
\bq\label{sqn}
\begin{aligned}
\dist(\mathcal E, \Rn \setminus B(\mathcal E)) 
& \ge \dist (\mathcal E, \Rn \setminus \lambda \mathcal E) = \dist (P \mathbf B(0,1),\lambda P (\R^n \setminus \mathbf B(0,1))) 
\\
&\ge \dist ( \mathbf B(0,1),\lambda (\R^n \setminus \mathbf B(0,1))) \min_{i=1,\ldots,n} {p_i}  = (\lambda-1)\min_{i=1,\ldots,n} {p_i} >\sqrt{n}.
\end{aligned}
\eq 
Define the rational lattice $\Gamma = U^{-1} \Zn$. By the above argument it suffices to construct $(A,\Gamma)$-MRA with a scaling vector in the Schwartz class.

Consider the orthonormal SFS basis obtained by applying the dilation operator $D_U$ to the standard SFS basis from Theorem \ref{wiln}
\begin{equation}\label{du}
E^{\Gamma}(D_U(\mathcal W_n(\phi)))= D_U(E^{\Zn}(\mathcal W_n(\phi)))= \{ T_{\gamma} D_U w_{\mathbf j}: \gamma \in \Gamma, \ \mathbf j \in \N_0^n \}.
\end{equation}
For $j \ge 0$, we let 
\[
I_j = (-(j+1)/2-\delta,-j/2+\delta) \cup (j/2-\delta,(j+1)/2+\delta).
\]
By \eqref{supp}, for any $\mathbf j=(j_1,\ldots j_n) \in \N_0^n$,
\bq\label{supp2}
\supp \widehat{ D_U w_{\mathbf j}} = \supp D_U \widehat{w_{\mathbf j}} = U^{\top}(\supp \widehat w_{\mathbf j})
\subset
U^{\top}(I_{\mathbf j}), \quad\text{where } I_{\mathbf j}= I_{j_1} \times \ldots \times I_{j_n}.
\eq
Define $\mathbf J \subset \N_0^n$ by
\bq\label{Jd}
\mathbf J = \{\mathbf j \in  \N_0^n: \mathcal E \cap U^{\top}(I_{\mathbf j}) \ne \emptyset \}.
\eq
Clearly, $\mathbf J$ is a finite set. Define $\Phi$, which we shall later prove to be a scaling vector of an MRA, and the corresponding core space $V_0$, by
\bq\label{V0}
\Phi = \{ D_U  w_{\mathbf j}: \mathbf j \in \mathbf J \}, \qquad
V_0 = \ov{\spa} \{ T_\ga \vp: \ga \in \Ga, \ \vp \in \Phi \}.
\eq
By Theorem \ref{wiln} and the identity \eqref{du}, $E^\Ga(\Phi)$ is an orthonormal basis of its closed linear span $V_0$.

\begin{figure}\label{figure2}

\begin{tikzpicture}[scale=1.5]

  \draw[step=.5cm,gray,very thin] (-2.4,-3.4) grid (2.4,3.4);
 % \draw (-1.5,0) -- (1.5,0);
  %\draw (0,-1.5) -- (0,1.5);
   \draw (0,0) ellipse [x radius=1.2, y radius=1.7];
   \draw[dashed] (0,-1.7) -- (0,1.7);
   \draw[dashed] (-1.2,0) -- (1.2,0);
   
   \draw[rotate=-10] (0,0) ellipse [x radius=2, y radius=3];
   \draw[rotate=-10, dashed] (0,-3) -- (0,3);
  \draw[rotate=-10, dashed]  (-2,0) -- (2,0);

   \draw[blue!50!black, thick] (-0.5,-0.5) rectangle (0.5,0.5);
   \draw (0,0) node {$\widehat w_{00}$};
   
   \draw[blue!50!black, thick]  (-0.5,-1) rectangle (0.5,-0.5);
   \draw (0,-.75) node {$\widehat w_{01}$};
   \draw[blue!50!black, thick]  (-0.5,0.5) rectangle (0.5,1);
   \draw (0,.75) node {$\widehat w_{01}$};

   \draw[blue!50!black, thick]  (-0.5,-1.5) rectangle (0.5,-1);
   \draw (0,-1.25) node {$\widehat w_{02}$};
   \draw[blue!50!black, thick]  (-0.5,1) rectangle (0.5,1.5);
   \draw (0,1.25) node {$\widehat w_{02}$};
   
   \draw[blue!50!black, thick]  (-0.5,-2) rectangle (0.5,-1.5);
   \draw (0,-1.75) node {$\widehat w_{03}$};
   \draw[blue!50!black, thick]  (-0.5,1.5) rectangle (0.5,2);
   \draw (0,1.75) node {$\widehat w_{03}$};
   
   \draw[blue!50!black, thick]  (-1,-0.5) rectangle (-0.5,0.5);
   \draw (-.75,0) node {$\widehat w_{10}$};
   \draw[blue!50!black, thick]  (0.5,-0.5) rectangle (1,0.5);
   \draw (.75,0) node {$\widehat w_{10}$};   
   
   \draw[blue!50!black, thick]  (-1.5,-0.5) rectangle (-1,0.5);
   \draw (-1.25,0) node {$\widehat w_{20}$};
   \draw[blue!50!black, thick]  (1,-0.5) rectangle (1.5,0.5);
   \draw (1.25,0) node {$\widehat w_{20}$};   
   
    \draw[blue!50!black, thick]  (-1,-1) rectangle (-.5,-.5);
    \draw (-.75,-.75) node {$\widehat w_{11}$};
    \draw[blue!50!black, thick]  (0.5,-1) rectangle (1,-.5);
    \draw (.75,-.75) node {$\widehat w_{11}$};
     \draw[blue!50!black, thick]  (-1,0.5) rectangle (-.5,1);
    \draw (-.75,.75) node {$\widehat w_{11}$};
    \draw[blue!50!black, thick]  (0.5,0.5) rectangle (1,1);
    \draw (.75,.75) node {$\widehat w_{11}$};
    
     \draw[blue!50!black, thick]  (-1.5,-1) rectangle (-1,-.5);
    \draw (-1.25,-.75) node {$\widehat w_{21}$};
    \draw[blue!50!black, thick]  (1,-1) rectangle (1.5,-.5);
    \draw (1.25,-.75) node {$\widehat w_{21}$};
     \draw[blue!50!black, thick]  (-1.5,0.5) rectangle (-1,1);
    \draw (-1.25,.75) node {$\widehat w_{21}$};
    \draw[blue!50!black, thick]  (1,0.5) rectangle (1.5,1);
    \draw (1.25,.75) node {$\widehat w_{21}$};
    
    \draw[blue!50!black, thick]  (-1,-1.5) rectangle (-.5,-1);
    \draw (-.75,-1.25) node {$\widehat w_{12}$};
    \draw[blue!50!black, thick]  (0.5,-1.5) rectangle (1,-1);
    \draw (.75,-1.25) node {$\widehat w_{12}$};
     \draw[blue!50!black, thick]  (-1,1) rectangle (-.5,1.5);
    \draw (-.75,1.25) node {$\widehat w_{12}$};
    \draw[blue!50!black, thick]  (0.5,1) rectangle (1,1.5);
    \draw (.75,1.25) node {$\widehat w_{12}$};
    
    \draw[blue!50!black, thick]  (-1,-2) rectangle (-.5,-1.5);
    \draw (-.75,-1.75) node {$\widehat w_{13}$};
    \draw[blue!50!black, thick]  (0.5,-2) rectangle (1,-1.5);
    \draw (.75,-1.75) node {$\widehat w_{13}$};
     \draw[blue!50!black, thick]  (-1,1.5) rectangle (-.5,2);
    \draw (-.75,1.75) node {$\widehat w_{13}$};
    \draw[blue!50!black, thick]  (0.5,1.5) rectangle (1,2);
    \draw (.75,1.75) node {$\widehat w_{13}$};
    
\end{tikzpicture}

\caption{Ellipsoid $\mathcal E$ and its image under expansive dilation $B$ superimposed with the supports of the SFS basis generators $\widehat w_{(j_1,j_2)}$ constituting the scaling vector $\Phi$.}
\end{figure}
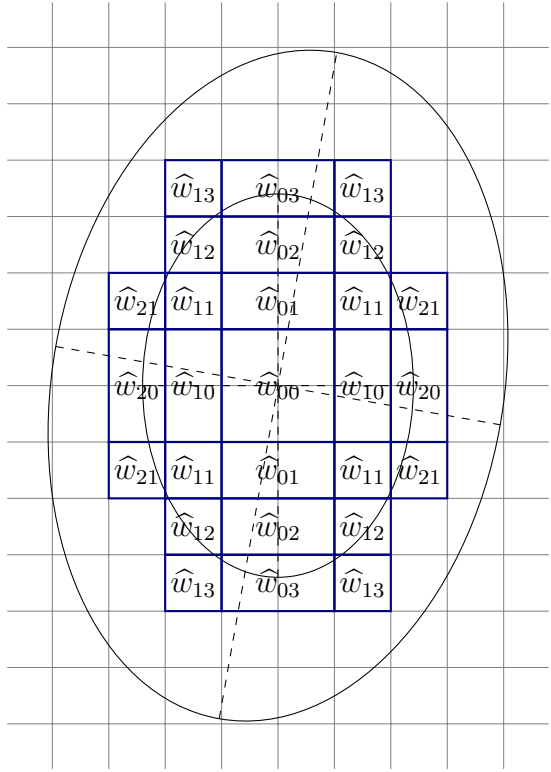

Define the spaces $V_j$ by $V_j=D_{A^j}(V_0)$, $j\in\Z$. Thus, properties (M4) and (M5) hold automatically by the definition. To verify the remaining MRA properties, we need to make some additional observations. By Theorem \ref{wiln}, \eqref{supp2}, \eqref{Jd}, and \eqref{V0}, we have
\bq\label{mra5}
\check L^2(\mathcal E) =  \{ f \in L^2(\Rn): \supp \widehat f \subset \mathcal E\}  \subset V_0.
\eq
Indeed, by \eqref{supp2} and \eqref{Jd}, any $f\in \check L^2(\mathcal E)$ satisfies 
\[
\lan f, T_\gamma D_U w_{\mathbf j} \ran =0 \qquad\text{for all } \ga\in \Ga, \ \mathbf j \in \N_0^n \setminus \mathbf J.
\]
On the other hand, we claim that 
\bq\label{mra6}
V_0 \subset \check L^2(B(\mathcal E)) = \{ f \in L^2(\Rn): \supp \widehat f \subset B(\mathcal E)\}.
\eq
To show \eqref{mra6} we need a careful geometric argument using  \eqref{sqn}. Note that each set $I_{\mathbf j}$ is the union of $2^n$ cubes of the form 
\bq\label{mra7}
\pm (j_1/2-\delta,(j_1+1)/2+\delta) \times \ldots \times \pm (j_n/2-\delta,(j_n+1)/2+\delta)
\eq
for all possible $2^n$ choices of signs $\pm$. Thus, if $\mathbf j \in J$, then by \eqref{Jd}, $P(\mathbf B(0,1)) \cap I_{\mathbf j} \ne\emptyset$. Thus, the ellipsoid $P(\mathbf B(0,1))$ must non-trivially intersect with one of the cubes of the form \eqref{mra7}. Since $P(\mathbf B(0,1))$ is symmetric with respect to coordinate axes, it must intersect non-trivially with all cubes of the form \eqref{mra7}. Now, each such cube has diameter $(1/2+2\delta)\sqrt{n}\le \sqrt{n}$ because $\delta\le 1/4$. Thus, each set $I_{\mathbf j}$, $\mathbf j \in \mathbf J$, is contained in $\sqrt{n}$-neighborhood of the ellipsoid $P(\mathbf B(0,1))$. By \eqref{supp2}, this implies that 
$\supp D_U \widehat w_{\mathbf j}$ is contained in $\sqrt{n}$-neighborhood of the ellipsoid $\mathcal E$, which in turn is contained in $B(\mathcal E)$ in light of \eqref{sqn}. This shows \eqref{mra6}.

The verification of MRA properties (M1)--(M3) is now a routine. By \eqref{mra5} and \eqref{mra6}, we show (M1)
\[
V_j \subset D_{A^j}(\check L^2(B(\mathcal E)))= \check L^2(B^{j+1}(\mathcal E)) = D_{A^{j+1}}(\check L^2(\mathcal E)) \subset V_{j+1}  \qquad\text{for } j\in \Z.
\]
Since $B$ is an expansive matrix we have $\bigcap_{j\in \Z} B^j(\mathcal E) = \{0\}$ and $\bigcup_{j\in \Z} B^j(\mathcal E) = \Rn$. This implies (M2) and (M3) and completes the proof of Theorem \ref{MRA}.
\end{proof}

\section{Constructing wavelets from rationally dilated MRAs}\label{S4}

In this section we describe a general construction procedure of a wavelet basis out of an MRA which is associated with  an expansive dilation matrix $A$ with rational entries. This subject has been studied by several authors in the setting of integer dilations \cite{B2, CHM, CMX, DiR, Gro0, Han, KLR, PR, Str, Woj}.
Unlike the well-studied case of dilations with integer entries, we need to work with a larger lattice $\Gamma$ than the standard lattice $\Z^n$. This is a consequence of the work of Hoover and the author on rationally dilated wavelet dimension function \cite{BH}.

By \cite[Lemma 4.1]{BH} a necessary condition for $(A,\Ga)$-MRA $\{V_j\}_{j\in\Z}$ to yield an orthogonal (or more generally, semi-orthogonal) wavelet is that $V_0$ is shift invariant under shifts of the lattice $A\Zn+\Zn$. In particular, if $A$ has rational entries, then a lattice $\Ga$ must necessarily contain $A\Zn+\Zn$ as its sublattice. In the special case when $A$ has integer entries, this merely means that the space $V_0$ is $\Z^n$-SI. However, for general rational $A$ it puts an additional constraint on the degree of shift-invariance of $V_0$.
In addition, if we require that our wavelet $\Psi$ satisfies a mild regularity (continuity and decay of $\widehat \Psi$), then by \cite[Theorem 6.4]{BH}, this imposes constraints also on the multiplicity of the MRA $\{V_j\}_{j\in\Z}$. Hence, it should not come as a surprise that a successful construction of even mildly regular wavelets for rational dilations requires a careful coordination of the choice of the lattice $\Ga$ and to a lesser degree the multiplicity of an MRA. 
Consequently, in this section we shall assume that we are given $(A,\Ga)$-MRA $\{V_j\}_{j\in\Z}$ with the lattice $\Ga=A\Zn+\Zn$ with a multiplicity $N\in N$. The existence of such MRA with a scaling vector function $\Phi$ in the Schwartz class is guaranteed by Theorem \ref{MRA}. 

Once this framework is adopted, the construction of rationally dilated wavelets follows a similar scheme as that for integer dilations. The refinability of the scaling vector $\Phi$ implies the existence of a matrix-valued low-pass filter $M$. If $A$ is an integer dilation, then $M$ has values in square $N\times N$ matrices. However, if $A$ is a rational dilation, then $M$ takes values in rectangular $N |\Gamma/\Z^n| \times N$ matrices. The corresponding wavelets are then constructed by appropriate choice of a matrix-valued low-pass filter $H$. If $A$ is an integer dilation, then $H$ has values in $N(|\det A|-1) \times N$ matrices. However, for rational dilations $A$, the values of $H$ are matrices of size $N(|\Ga/(A\Zn)|-|\Ga/\Zn|) \times N$. The commonality between integer and rational dilations is that the orthonormality of the resulting wavelets is equivalent to almost everywhere orthogonality of a square matrix of the form 
\[
\frac{1}{\sqrt p}
\begin{bmatrix} 
M(\xi+\omega_1) & M(\xi+\omega_2) & \ldots & M(\xi+\omega_p)
\\
H(\xi+\omega_1) & H(\xi+\omega_2) & \ldots & H(\xi+\omega_p)
\end{bmatrix},
\]
where $\omega_1,\ldots,\omega_p$ are cosets representatives of the quotient group $(A\Z^n)^*/\Ga^*$. Here, given a full rank lattice $\Gamma$, we define its dual lattice by
\[
\Gamma^*= \{ k\in \R^n: \langle k, \gamma \rangle \in \Z \quad\text{for all }\gamma \in \Gamma\}.
\]

We start with a basic result on finitely generated shift-invariant spaces \cite[Theorem 1.7]{BDR1} that follows from Helson's characterization of SI spaces in terms of range functions \cite[Proposition 1.5]{B0}. Note that Lemma \ref{L1} does not require any stability assumptions on the shift-invariant system generated by $\Phi$, though we will later assume that $E^\Ga(\Phi)$ is an orthonormal sequence. Given $K \times N$ matrix $M=(m_{i,j})_{i\in [K], j\in [N]}$, its Frobenius norm is defined as
\[
||M||_F = \sqrt{ \tr{(M^* M)}} = \bigg( \sum_{i=1}^K \sum_{j=1}^N |m_{i,j}|^2 \bigg)^{1/2}.
\]

\begin{lemma}\label{L1}
Suppose that $V_0=\ov{\spa}{E^\Ga(\Phi)}$, where $\Phi=\{\vp^1, \ldots, \vp^N \} \subset L^2(\Rn)$ and $\Ga$ is a full rank lattice. Let $f\in L^2(\Rn)$. Then $f \in V_0$ $\iff$ there exists a measurable $\Ga^*$-periodic function $M_f: \Rn \to M_{1\times N}(\C)$ such that:
\begin{equation}\label{L2}
\widehat f(\xi) = M_f(\xi) \widehat \Phi(\xi)\qquad\text{for a.e. }\xi \in \R^n,\text{ where } 
\widehat \Phi(\xi) = \begin{bmatrix} \widehat \vp^1(\xi) \\
 \vdots \\ \widehat \vp^N(\xi)
\end{bmatrix}.
\end{equation}
In addition, if $E^\Gamma(\Phi)$ is an orthonormal basis of $V_0$, then 
\[
||f||^2 = \vol(\Ga) \int_{\Rn/\Ga^*} ||M_f(\xi)||_F^2 d\xi \qquad\text{for all }f\in V_0.
\]
\end{lemma}

\begin{proof}
A standard scaling argument reduces the proof to the setting of the standard lattice case $\Ga=\Zn$. Then, \eqref{L2} follows from \cite[Theorem 1.7]{BDR1}, albeit a function $M_f$, which is in general far from being unique, is not guaranteed to be measurable. To achieve the measurability of $M_f$, we need to use an inductive argument on the number of generators in $\Phi$. 

If $N=1$, then $M_f$ is uniquely determined, and hence measurable, on the spectrum of $V_0$ defined as 
\[
\sigma(V_0)= \{\xi \in \T^n: \dim J(\xi) >0 \},
\]
where $J$ is the range function corresponding to $V_0$ in Helson's characterization of SI spaces, see \cite[Proposition 1.5]{B0}. Outside of the spectrum of $V_0$, $M_f$ can take arbitrary values, so to guarantee its measurability we shall impose that $M_f(\xi)=0$ for $\xi \not \in \sigma(V_0)$. 

Suppose that Lemma \ref{L1} holds for SI space $V_0$ generated by $\Phi=\{\vp^1,\ldots,\vp^N\}$, $N\ge 1$. Let $V'_0$ be a SI space generated by one more generator, that is, by $\Phi'=\Phi \cup \{\vp^{N+1}\}$. Let $P$ be the orthogonal projection onto principal shift-invariant space (PSI) generated by $\vp^{N+1}$. Take any $f\in V_0'$ and decompose it as $f=Pf+g$, where $g\in V_0$. By inductive hypothesis there exists a measurable $\Zn$-periodic $M_{1\times N}(\C)$-valued function $M_g$ such that $\widehat g(\xi)=M_g(\xi) \widehat \Phi(\xi)$. Likewise, $(P_Wf)\widehat{}(\xi)=m(\xi) \widehat \vp^{N+1}(\xi)$ for some  measurable $\Zn$-periodic scalar-valued $m$. Then,
\[
\widehat f(\xi) = M_f(\xi) \widehat \Phi'(\xi)\qquad\text{for a.e. }\xi \in \R^n,\text{ where } M_f(\xi) = [ M_g(\xi) \ m(\xi) ].
\]

A standard result about SI system says that $E^\Ga(\Phi)$ is an orthonormal sequence $\iff$ a Gramian matrix of its fiberization vectors is the identity a.e. (modulo a suitable normalization factor). That is, 
\begin{equation}\label{L5a}
\sum_{k\in \Ga^*} \widehat\Phi(\xi+k)\widehat \Phi^*(\xi+k) = \vol(\Ga) \mathbf I_N  \qquad\text{a.e. }\xi.
\end{equation}
Here, $\vol(\Ga)$ is the volume of the lattice $\Ga$ defined as the measure of a fundamental domain of $\Rn/\Ga$ and $\mathbf I_N$ is an $N\times N$ identity matrix. 
Hence, if $E^\Gamma(\Phi)$ is an orthonormal basis of $V_0$, then for any $f\in V_0$ we have
\[
\begin{aligned}
||f||^2 = \int_{\Rn} M_f(\xi) \widehat \Phi(\xi) \widehat \Phi(\xi)^* M_f(\xi)^* d\xi 
&= \int_{\Rn/\Ga^*} M_f(\xi) \bigg( \sum_{k\in \Gamma^*} \widehat \Phi(\xi+k) \widehat \Phi(\xi+k)^* \bigg) M_f(\xi)^* d\xi 
\\
& = \vol(\Ga) \int_{\Rn/\Ga^*} ||M_f(\xi)||_F^2 d\xi . \hskip3.3cm \qedhere
\end{aligned}
\]

\end{proof}

The following lemma shows a generalization of the fundamental equation satisfied by a low-pass filter corresponding to a scaling function of an MRA. In this context the identity \eqref{L5} is often referred to as the Smith-Barnwell equation or the quadrature mirror equation. An original, and thus atypical, feature of Lemma \ref{L3} is that we do not require any refinability of the space $V_0$, though we will later assume that $V_0$ is refinable.

\begin{lemma}\label{L3}
Let $\Lambda$ and $\Gamma$ be two full rank lattice such that $\Lambda \subset \Gamma$.
Suppose that $E^\Gamma(\Phi)$ is an orthonormal basis of its closed linear span $V_0$, where $\Phi=\{\vp^1, \ldots, \vp^N \} \subset L^2(\Rn)$. Let $\Psi=\{\psi^1, \ldots, \psi^{K} \} \subset V_0$. Let $M_\Psi: \Rn \to M_{K \times N}(\C)$ be a measurable $\Ga^*$-periodic function such that
\begin{equation}\label{L4}
\widehat \Psi(\xi) = M_\Psi(\xi) \widehat \Phi(\xi)\qquad\text{for a.e. }\xi \in \R^n,
\text{ where } 
\widehat \Psi(\xi) = \begin{bmatrix} \widehat \psi^1(\xi) \\
 \vdots \\ \widehat \psi^K(\xi)
\end{bmatrix}.
\end{equation}
Let $\{\omega_1,\ldots,\omega_p\}$ be a transversal set of $\Lambda^*/\Ga^*$, i.e., representatives of distinct cosets, and let $p=|\Lambda^*/\Ga^*|$.
Then, $E^{\Lambda}(\Psi)$ is an orthonormal sequence $\iff$
\begin{equation}\label{L5}
\sum_{j=1}^p M_\Psi(\xi+\omega_j) (M_\Psi)^*(\xi+\omega_j)=p \mathbf I_K
\qquad\text{for a.e. }\xi \in \R^n.
\end{equation}
Moreover, $E^{\Lambda}(\Psi)$ is an orthonormal basis of $V_0$ $\iff$ $K=N p$ and \eqref{L5} holds.
\end{lemma}

\begin{proof}
The existence of function $M_\Psi$ is a consequence of Lemma \ref{L1}.
Recall that $E^{\Lambda}(\Psi)$ is an orthogonal sequence $\iff$ 
\[
 \sum_{k\in \Lambda^*} \widehat\Psi(\xi+k)\widehat \Psi^*(\xi+k) = \vol(\Lambda)\mathbf I_K \qquad\text{for a.e. }\xi.
 \]
On the other hand,
\begin{align*}
 \sum_{k\in \Lambda^*} & \widehat\Psi(\xi+k)\widehat \Psi^*(\xi+k)
 =
\sum_{k\in \Lambda^*} M_\Psi(\xi+k) \widehat\Phi(\xi+k)\widehat \Phi^*(\xi+k)(M_\Psi)^*(\xi+k)
\\
&= \sum_{j=1}^p \sum_{k\in \Ga^*} 
 M_\Psi(\xi+\omega_j+k) \widehat\Phi(\xi+\omega_j+k)\widehat \Phi^*(\xi+\omega_j+k)(M_\Psi)^*(\xi+\omega_j+k)
\\
&= \sum_{j=1}^p 
 M_\Psi(\xi+\omega_j) \bigg( \sum_{k\in \Ga^*} \widehat\Phi(\xi+\omega_j+k)\widehat \Phi^*(\xi+\omega_j+k) \bigg) (M_\Psi)^*(\xi+\omega_j)
 \\
 &= \vol(\Gamma) \sum_{j=1}^p M_\Psi(\xi+\omega_j) (M_\Psi)^*(\xi+\omega_j).
\end{align*}
In the above calculation we have used \eqref{L4}, $\Ga^*$-periodicity of $M_\Psi$, and \eqref{L5a}. Observe also that 
\begin{equation}\label{L8z}
p=\vol(\Gamma^*)/\vol(\Lambda^*)= \vol(\Lambda)/\vol(\Ga),
\end{equation}
Combining these results yields \eqref{L5}.

Given $\Ga$-SI space $V \subset \R^n$, we recall that the dimension function $\dim^\Ga_V: \R^n \to \N \cup \{0,\infty\}$ can be defined as the multiplicity function of the projection-valued measure coming from the representation of $\Ga$ on $V$ via translations by Stone's Theorem \cite{Ba1, BMM}. Alternatively, one can use Helson's characterization of shift-invariant spaces in terms of range functions. A version of this result for general lattices $\Gamma \subset \R^n$ can be found in the work of Rzeszotnik and the author \cite[Proposition 1.1]{BR2}. Hence, the dimension function can be defined as
\[
\dim^\Ga_V(\xi) = \dim \overline{\spa} \{ (\widehat \vp(\xi+k))_{k\in\Ga^*}: \vp \in \mathcal A \},
\]
where $\mathcal A \subset V$ is a countable set of generators of $V$, i.e., $V = \ov{\spa} E^\Ga(\mathcal A)$.

By the assumption $E^\Ga(\Phi)$ is an orthonormal basis of $\Ga$-SI space $V_0$. Hence, $\dim^\Ga_{V_0}(\xi) \equiv N$ a.e. $\xi$. Thus, by \cite[Lemma 2.5]{BH} 
\[
\dim_{V_0}^{\Lambda}(\xi) = \sum_{j=1}^{p} \dim_{V_0}^{\Ga}(\xi+\omega_j) = \sum_{j=1}^p N = Np \qquad\text{a.e. }\xi.
\]
Consequently, an orthonormal sequence $E^{\Lambda}(\Psi)$ is complete in $V_0$ $\iff$ the number of its generators $K=Nb$.
\end{proof}

The condition \eqref{L5} can be equivalently stated in terms of polyphase matrices defined below. A polyphase decomposition is a well-known technique in the theory of filter banks. A relative novelty of Lemma \ref{L6} is that we phrase it in the setting of two  lattices, one of which is a sublattice of the other. %For convenience we assume that the sparser lattice is the standard lattice $\Zn$, though it can be replaced by any other full rank lattice. 

\begin{lemma}\label{L6}
Let $\Lambda$ and $\Gamma$ be two full rank lattice such that $\Lambda \subset \Gamma$.
Let $\{d_1,\ldots,d_p\}$ be a transversal set of $\Ga/\Lambda$, where $p=|\Ga/\Lambda|=|\Lambda^*/\Ga^*|$.
Suppose that $M: \Rn \to M_{K\times N} (\C)$ is a measurable $\Ga^*$-periodic  function such that 
\[
\int_{\Rn/\Ga^*} ||M(\xi)||_F^2 d\xi<\infty.
\]
Then, there exists a unique collection of $\Lambda^*$-periodic functions $M^{\uparrow d_j}: \Rn \to M_{K\times N} (\C)$ such that 
\[
\int_{\Rn/\Lambda^*} ||M^{\uparrow d_j}(\xi)||_F^2 d\xi<\infty
\]
and
\begin{equation}\label{L7}
M(\xi) = \sum_{j=1}^p e^{2\pi i \lan d_j, \xi \rangle } M^{\uparrow d_j}(\xi) 
\qquad\text{a.e. }\xi.
\end{equation}
Moreover, we have
\begin{equation}\label{L8}
\frac{1}{p} \int_{\Rn/\Gamma^*} ||M(\xi)||_F^2 d\xi = \sum_{j=1}^p \int_{\Rn/\Lambda^*} ||M^{\uparrow d_j}(\xi)||_F^2 d\xi.
\end{equation}
\end{lemma}

\begin{proof}
By Parseval's Theorem we can expand each entry of a matrix-valued function $M$ in terms of Fourier series
\[
M(\xi) = \sum_{\ga \in \Ga} a(\ga) e^{2\pi i \lan \ga, \xi \ran},
\qquad\text{where }a(\ga) = 
\frac{1}{\vol(\Ga^*)}
\int_{\Rn/\Ga^*} M(\xi) e^{-2\pi i \lan \ga, \xi \ran} d\xi \in M_{K\times N} (\C).
\]
Moreover,
\begin{equation}\label{L8a}
\frac{1}{\vol(\Ga^*)} \int_{\Rn/\Ga^*} ||M(\xi)||_F^2 d\xi = \sum_{\ga \in \Ga} ||a(\gamma)||^2_F.
\end{equation}
Since every $\ga \in \Ga$ can be uniquely written as $\ga = d_j+k$ for some $j=1,\ldots,p$ and $k\in\Lambda$, we have
\[
M(\xi)= \sum_{j=1}^p e^{2\pi i \lan d_j, \xi \ran} \bigg(\sum_{k\in \Lambda} a(d_j+k) e^{2\pi i \lan k, \xi \ran}\bigg).
\]
The Fourier series in the parenthesis defines $K \times N$ matrix valued and $\Lambda^*$-periodic function, which we denote by $M^{\uparrow d_j}$, satisfying \eqref{L7}. By Parseval's  Theorem
\begin{equation}\label{L8b}
\frac{1}{\vol(\Lambda^*)} \int_{\Rn/\Lambda^*} ||M^{\uparrow d_j}(\xi)||_F^2 d\xi = \sum_{k\in \Lambda} ||a(d_j + k)||^2_F.
\end{equation}
 Reversing the steps of this argument shows the uniqueness of functions $M^{\uparrow d_j}$. Combining \eqref{L8z}, \eqref{L8a}, and \eqref{L8b} yields \eqref{L8}.
\end{proof}

Lemma \ref{L9} is a variant of Lemma \ref{L3} phrased in terms of polyphase matrices. Again, it is worth emphasizing that we do not require any refinability conditions at this stage, yet.

\begin{lemma}\label{L9}
Under the assumptions of Lemma \ref{L3}, let $M_\Psi^{\uparrow d_j}$ be the polyphase matrix functions corresponding to the function $M_\Psi$ as in \eqref{L4}. Then, $E^{\Lambda}(\Psi)$ is an orthonormal basis of $V_0$ $\iff$ $K=N p$ and 
\begin{equation}\label{L10}
\sum_{j=1}^p M_\Psi^{\uparrow d_j}(\xi) (M_\Psi^{\uparrow d_j})^*(\xi)= \mathbf I_K
\qquad\text{for a.e. }\xi \in \R^n.
\end{equation}
\end{lemma}

\begin{proof}
For simplicity of notation we shall drop the dependence on $\Psi$ for functions $M_\Psi$ and $M_\Psi^{\uparrow d_j}$.
The lemma is a consequence of the following calculation that uses the notation from Lemma \ref{L6}
\begin{align*}
\sum_{j=1}^p M(\xi+\omega_j) M^*(\xi+\omega_j)
& =  \sum_{j_0=1}^p \sum_{j_1=1}^p \bigg( \sum_{j=1}^p e^{2\pi i \lan d_{j_0}-d_{j_1}, \xi+\omega_j \ran } \bigg) M^{\uparrow d_{j_0}}(\xi) (M^{\uparrow d_{j_1}})^*(\xi)
\\
& = p \sum_{j=1}^p M^{\uparrow d_{j}}(\xi) (M^{\uparrow d_{j}})^*(\xi).
\end{align*}
Here, we have used the identity
\[
\sum_{j=1}^p e^{2\pi i \lan d_{j_0}-d_{j_1}, \omega_j \ran } = p \delta_{j_0,j_1}.
\]
This is a consequence of the duality of quotient groups $\Ga/\Lambda$ and $\Lambda^*/\Ga^*$, see \cite[Lemma 3.6]{BL}.
Consequently, \eqref{L5} is equivalent to \eqref{L10}, which completes the proof of the lemma.
\end{proof}

Our next goal is to formulate main result of this section, which is a criterion on the existence of orthogonal wavelet associated with rationally dilated MRA. At this moment we make the {\bf standing assumptions} that was motivated in the beginning of this section:
\begin{itemize}\label{stand}
\item $A$ is an $n\times n$ expansive rational dilation and  $\Ga= A\Zn+\Zn$,
\item $\{ V_j \}_{j\in \Z}$ is a $(A,\Gamma)$-multiresolution analysis of multiplicity $N\in \N$, and  
\item the scaling vector function $\Phi=\{\vp^1, \ldots, \vp^N \} \subset V_0$ generates an orthonormal basis $E^\Gamma(\Phi)=\{T_\gamma \vp: \gamma \in \Gamma, \ \vp \in \Phi\}$ of $V_0$.
\end{itemize}

As a consequence of Lemma \ref{L1} we have the following result.

\begin{lemma}\label{v0}
A function $f\in L^2(\Rn)$ belongs to $V_1=D_A(V_0)$ if and only if
\begin{equation}\label{v1}
b^{1/2} \widehat f(B\xi)= M_f(\xi) \widehat \Phi(\xi)
\qquad\text{for a.e. }\xi \in \R^n,
\end{equation}
for some $\Ga^*$-periodic measurable function with values in $M_{1\times N}(\C)$, where $B=A^\top$ and $b:=|\det A|=|\det B|$.
\end{lemma}

\begin{proof}
Note that  a function $f\in L^2(\Rn)$ belongs to $V_1=D_A(V_0)$ iff $D_{A^{-1}} f \in V_0$. Note that $\widehat{(D_{A^{-1}} f)}(\xi)= |\det B|^{1/2} \widehat f(B\xi) = D_B \widehat f(\xi)$. Hence, by Lemma \ref{L1} there exists a measurable $\Ga^*$-periodic function $M_f: \Rn \to M_{1\times N}(\C)$ such that \eqref{v1} holds.
\end{proof}

To formulate our result, we need to set up some necessary notation. Let $\{k_1,\ldots,k_q\}$ be a transversal set of the quotient group $\Ga/\Zn$ and $q:=|\Ga/\Zn|$. Define $\tilde \Phi$ to be the column vector consisting of functions
\begin{equation}\label{p0}
\tilde \Phi = \{ T_{k_j} \vp^l: j=1,\ldots,q, \ l=1,\ldots, N \}.
\end{equation}
By Lemma \ref{v0}, there exists a measurable $qN \times N$ matrix-valued function $M$ on $\Rn$ such that $M$ is $\Ga^*$-periodic and  
\begin{equation}\label{p1}
b^{1/2}\widehat {\tilde \Phi}(B\xi) = M(\xi) \widehat \Phi(\xi)  \qquad\text{a.e. }\xi.
\end{equation}
We shall refer to $M$ as the low-pass filter corresponding to the scaling vector function $\Phi$. Indeed, in the case when the dilation $A$ has integer entries, the function $M$ has values in $N\times N$ matrices and satisfies a familiar refinability equation
\[
b^{1/2} \widehat \Phi(B\xi) = M(\xi) \widehat \Phi(\xi) \qquad\text{a.e. }\xi.
\]

Let $\Psi=\{\psi^1,\ldots,\psi^L\}$ be any finite collection of functions  in $V_1$. By Lemma \ref{v0}, there exists $L \times N$ matrix-valued function $H$ on $\Rn$ such that $H$ is $\Ga^*$-periodic and
\begin{equation}\label{p2}
b^{1/2} \widehat {\Psi}(B\xi) = H(\xi) \widehat \Phi(\xi)  \qquad\text{a.e. }\xi.
\end{equation}
We shall refer to $H$ as the high-pass filter corresponding to the vector function $\Psi$. 

Define the lattice $\Lambda:=A\Z^n$. Let $\{d_1,\ldots,d_p\}$ be a transversal set of the quotient group $\Ga/\Lambda$ and $p:=|\Ga/\Lambda|$. 
Let $M^{\uparrow d_j}$, $j=1,\ldots, p$, be the polyphase $qN \times N$ matrix functions corresponding to the low-pass filter $M$, which are given by Lemma \ref{L6}. Likewise, let $H^{\uparrow d_j}$, $j=1,\ldots, p$, be the polyphase $L \times N$ matrix functions corresponding to the high-pass filter $H$.

\begin{theorem}\label{cri}
Let $A$ be an expansive rational dilation and let $\Ga= A\Zn+\Zn$. Suppose that there exists $\Phi=\{\vp^1, \ldots, \vp^N \} \subset V_0$ such that 
$E^\Gamma(\Phi)$ is an orthonormal basis of $V_0$. In addition, suppose that $V_0$ is a refinable space with respect to $A$, i.e., $V_0 \subset V_1:=D_A(V_0)$. Let $\Psi=\{\psi^1,\ldots, \psi^L\}$ be a finite collection in $V_1$. Let $M$ and $H$ be low-pass and high-pass filters as in \eqref{p1} and \eqref{p2}.

Then,  $E^{\Zn}(\Psi)$ is an orthonormal basis of $V_1 \ominus V_0$ if and only if
\[
L=N(|\Ga/(A\Zn)|-|\Ga/\Zn|) = N(p-q),
\]
and 
\begin{equation}\label{cri3}
\sum_{j=1}^p 
\begin{bmatrix} M^{\uparrow d_j}(\xi) \\ H^{\uparrow d_j}(\xi) \end{bmatrix}
\begin{bmatrix} M^{\uparrow d_j}(\xi) \\ H^{\uparrow d_j}(\xi) \end{bmatrix}^*
%\begin{bmatrix} (M^{\uparrow d_j})^*(\xi)  & (H^{\uparrow d_j})^*(\xi) \end{bmatrix}
= \mathbf I_{pN}
\qquad\text{for a.e. }\xi \in \R^n.
\end{equation}
\end{theorem}

Observe that the condition \eqref{cri3} equivalent to the  $(qN +L) \times pN$ matrix
\begin{equation}\label{cri4}
\begin{bmatrix} 
M^{\uparrow d_1}(\xi) & M^{\uparrow d_2}(\xi) & \ldots & M^{\uparrow d_p}(\xi)
\\
H^{\uparrow d_1}(\xi) & H^{\uparrow d_2}(\xi) & \ldots &H^{\uparrow d_p}(\xi)
\end{bmatrix}.
\end{equation}
being unitary for a.e. $\xi$.

\begin{proof}[Proof of Theorem \ref{cri}]

Define $(L+qN)\times N$ matrix-valued function $T$ by letting
\[
T(\xi) = \begin{bmatrix} M(\xi) \\ H(\xi) \end{bmatrix} \qquad \xi \in \R^n.
\]
Then, equations \eqref{p1} and \eqref{p2} can be combined as
\[
\begin{bmatrix} b^{1/2} \widehat {\tilde \Phi}(B\xi) \\  b^{1/2} \widehat \Psi(B\xi) \end{bmatrix} = \begin{bmatrix} M(\xi) \\ H(\xi) \end{bmatrix} \widehat \Phi(\xi).
\]
Observe that the following statements are equivalent:
\begin{enumerate}[(i)]
\item the system $E^{\Zn}(\Psi)$ is an orthonormal basis of $V_1 \ominus V_0$,
\item the system  $E^{\Zn}(\tilde \Phi \cup \Psi)$ is an orthonormal basis of $V_1$,
\item $E^{A\Z^n}(D_{A^{-1}}\tilde \Phi \cup D_{A^{-1}}\Psi)$ being an orthonormal basis of $V_0 = D_{A^{-1}}(V_1)$.
\end{enumerate}
The last equivalence follows by applying the dilation operator $D_{A^{-1}}$, which is unitary.

Let  $\{\omega_1,\ldots,\omega_p\}$ be a transversal set of $\Lambda^*/\Ga^*$, where $\Lambda = A\Z^n$ and $\Ga= A\Zn+\Zn$. Applying Lemma \ref{L4} to the collection $D_{A^{-1}}\tilde \Phi \cup D_{A^{-1}}\Psi \subset V_0$ of $qN+L$ functions, shows that the above statements are equivalent with $qN+L=pN$ and
\[
\sum_{j=1}^p 
\begin{bmatrix} M(\xi+\omega_j) \\ H(\xi+\omega_j) \end{bmatrix}
\begin{bmatrix} M(\xi+\omega_j) \\ H(\xi+\omega_j) \end{bmatrix}^*
= p \mathbf I_{pN}
\qquad\text{for a.e. }\xi \in \R^n.
\]
In other words, the following block matrix is square and unitary for a.e. $\xi$,
\[
\frac{1}{\sqrt p}
\begin{bmatrix} 
M(\xi+\omega_1) & M(\xi+\omega_2) & \ldots & M(\xi+\omega_p)
\\
H(\xi+\omega_1) & H(\xi+\omega_2) & \ldots & H(\xi+\omega_p)
\end{bmatrix}.
\]
Applying Lemma \ref{L9} to the matrix-valued function $T$ completes the proof of Theorem \ref{cri}.
\end{proof}

The following corollary summarizes the construction procedure of orthonormal wavelets associated with rationally dilated MRA.

\begin{corollary}\label{LH}
Suppose that $\Phi=\{\vp^1, \ldots, \vp^N \}$ is a scaling vector function of $(A,\Gamma)$-MRA, which satisfies the standing assumptions. Let $M$ be the low-pass filter of $\Phi$ satisfying \eqref{p1} and let $M^{\uparrow d_i}$, $i=1,\ldots,p$, be the corresponding polyphase matrices as in \eqref{L7}.

Suppose that $qN \times pN$ matrix function $\begin{bmatrix} 
M^{\uparrow d_1}(\xi) & M^{\uparrow d_2}(\xi) & \ldots & M^{\uparrow d_p}(\xi) \end{bmatrix}$ 
can be appended  by $(p-q)N \times pN$ matrix function
$\begin{bmatrix} 
H^{\uparrow d_1}(\xi) & H^{\uparrow d_2}(\xi) & \ldots &H^{\uparrow d_p}(\xi)
\end{bmatrix}$,
which is measurable and $\Z^n$-periodic, such that the resulting matrix \eqref{cri4} is unitary for a.e. $\xi$. Let $H$ be $(p-q)N \times N$ matrix function, which is $\Ga^*$-periodic, such that $H^{\uparrow d_i}$'s are its polyphase decomposition:
\begin{equation}\label{H7}
H(\xi) = \sum_{j=1}^p e^{2\pi i \lan d_j, \xi \rangle } H^{\uparrow d_j}(\xi) 
\qquad\text{a.e. }\xi.
\end{equation}
Define $\Psi=\{\psi^1,\ldots, \psi^{(p-q)N}\} \subset L^2(\R^n)$ such that \eqref{p2} holds, i.e., $H$ is a high-pass filter of $\Psi$. Then, $\Psi$ is an orthonormal wavelet with respect to the dilation $A$ and the lattice $\Z^n$.
\end{corollary}

\begin{proof}
Suppose that $\Phi=\{\vp^1, \ldots, \vp^N \}$ is a scaling vector function of $(A,\Gamma)$-MRA $\{V_j\}_{j\in \Z}$ of multiplicity $N$, which satisfies the standing assumptions. In particular, this means the assumptions of Theorem \ref{cri} are met. Let $L=(p-q)N$. Suppose that $L \times pN$ matrix function $\begin{bmatrix} 
H^{\uparrow d_1}(\xi) & H^{\uparrow d_2}(\xi) & \ldots &H^{\uparrow d_p}(\xi)
\end{bmatrix}$ is measurable, $\Z^n$-periodic, and satisfies \eqref{cri3}. Define the collection $\Psi=\{\psi^1,\ldots,\psi^L\}$ by 
\[
b^{1/2} \widehat {\Psi}(B\xi) = H(\xi) \widehat \Phi(\xi)  \qquad\text{a.e. }\xi.
\]
where $H$ is $\Gamma^*$-periodic $L\times N$ matrix function given by \eqref{H7}. By Theorem \ref{cri}, the shift invariant system $E^{\Zn}(\Psi)$ is an orthonormal basis of $V_1 \ominus V_0$. Applying dilation operator $D_A$ we deduce that $D_{A^j}(E^{\Zn}(\Psi))$ is an orthonormal basis of the space $W_j=V_{j+1} \ominus V_j$ for every $j\in \Z$. By properties (M2) and (M3) of an MRA we have an orthogonal decomposition
\[
\bigoplus_{j\in \Z} W_j = L^2(\R^n).
\]
Consequently, the wavelet system
\[
\{D_{A^j}T_k \psi^l: j\in \Z,\ k\in \Z^n,\ l=1,\ldots,L\}
\]
is an orthonormal basis of $L^2(\R^n)$.
\end{proof}

\section{Matrix completion and existence of Meyer wavelets} \label{S5}

In this section we show the existence of wavelets in the Schwartz class which are associated with an expansive dilation matrix $A$ with rational entries. It is well-known that for dilations with integer entries, the existence of nice wavelets associated with an MRA is connected with the matrix completion problem. A low-pass filter associated to a scaling function of an MRA defines the first row of a matrix function. To construct an orthonormal wavelet we need to find suitable high-pass filter, which corresponds to completing the first row to a unitary matrix function. 

The results of previous section show that the same scheme also works for dilations $A$ with rational entries, albeit instead of a single row, several orthogonal row functions are determined by a single scaling function. This scenario already happens for  rationally dilated MRA of multiplicity $1$; it is reminiscent of the construction of wavelets which are associated with integer dilations and MRAs of higher multiplicity, see \cite[Section 4.2]{CHM}. The goal is to extend the prescribed row functions in a such a way so that the resulting matrix function is unitary. While measurable matrix extension always exists, there are topological obstructions for the existence of continuous matrix extensions, see \cite{BJ, CMX, DiR, PR}.

A typical result for the existence of orthogonal wavelets requires the existence of an MRA with $r$-regular scaling function $\vp$.

\begin{definition}
We say that a function $f$ on $\R^n$ is $r$-regular, if $f$ is of class $C^r$, $r=0,1,\ldots,\infty$ and
$$| \partial^\al f(x)| \le c_{\al,k} (1+|x|)^{-k},$$
for each $k\in\N$, and each multi-index $\al$, with $|\al|\le r$. A wavelet family $\Psi=\{\psi^1,\ldots,\psi^L\}$ is $r$-regular, if $\psi^1, \ldots, \psi^L$ are $r$-regular functions. An $(A,\Gamma)$ MRA $\{V_j\}_{j\in \Z}$ is $r$-regular if there exist $r$- regular scaling functions $\Phi=\{\vp^1, \ldots, \vp^N \} $ such that 
$
E^\Gamma(\Phi):=\{T_\gamma \vp: \gamma \in \Gamma, \ \vp \in \Phi\}
$
is an orthonormal basis of $V_0$.
\end{definition}

We start by describing what is known for dilations $A$ with integer entries. The first result in this direction is due to Gr\"ochenig \cite{Gro0} who showed the existence of $r$-regular wavelet bases for dyadic dilation $A=2\mathbf I$ by constructing unitary matrix completion using Schmidt-Gram orthogonalization. An explicit construction of unitary matrix completion given the first row, which uses Householder transformations, was shown by Jia and Micchelli \cite[Proposition 2.1]{JM}, see also \cite[Theorem 5.1]{JS}. Consequently, we have the following existence result, which can be found in the book of Wojtaszczyk \cite[Theorem 5.15]{Woj}. 

\begin{theorem}\label{2p}
Let $A$ be an $n\times n$ expansive integer dilation such that $2b-1>n$, where $b=|\det A|$. Then, for any $r$-regular $(A,\Z^n)$-MRA of multiplicity $1$ with scaling function $\vp$, there exists an orthonormal wavelet $\Psi=\{\psi^1,\ldots, \psi^{b-1}\}$ consisting of $r$-regular functions.
\end{theorem}

The condition  $2b-1 >n$ can be relaxed if either $b=|\det A|=2$ or the scaling function $\vp$ is such that $\widehat \vp$ is real-valued, see \cite[Corollary 5.17]{Woj}. However, Packer and Rieffel \cite{PR} have shown that for a general a continuous low-pass filter for dilation $A$ with $|\det A| > 2$, there are topological obstructions for construction of associated continuous high-pass filters. They gave an example  of a continuous low-pass filter for $5\times 5$ integer dilation matrix $A$ with $\det A=3$ for which it is impossible to construct a family of the corresponding continuous high-pass filters. Thus, the dimensional condition $2b-1 >n$ in Theorem \ref{2p} is needed starting with the dimension $n = 5$.

Several authors have also studied the existence of wavelets associated with an MRA of higher multiplicity.  The matrix completion problem for multivariate filter bank construction was considered in \cite{CMX}. For the construction of orthonormal wavelets from an MRA of higher multiplicity see \cite[Chapter 4]{CHM}. The extension of Theorem \ref{2p} for an MRA of higher multiplicity was shown by Cabrelli and Gordillo \cite[Theorem 4.4]{CG}.   Their result was possible due to the following theorem of Ashino and Kametani \cite[Theorem 1]{AK}, which extends Gr\"ochenig's construction of unitary matrix completion where multiple orthogonal rows are given.

\begin{theorem}\label{AK}
Let $X$ be a real, compact, $C^\infty$ manifold with $\dim X=n$. Let $m, n, d\in \N$ be such that $n \le 2(m-d)$. Then, for all $C^\infty$ mappings $f_l: X \to \C^m$, $l=1,\ldots,d $ such that
\[
\langle f_k(x), f_l(x) \rangle = \delta_{k,l} \qquad\text{for all }k,l=1,\ldots,d, \ x\in X,
\]
there exist $C^\infty$ mappings $f_l: X \to \C^m$, $l=d+1,\ldots,m $ such that
\[
\langle f_k(x), f_l(x) \rangle = \delta_{k,l} \qquad\text{for all }k,l=1,\ldots,m, \ x\in X.
\]
\end{theorem}

Using Theorem \ref{AK} we prove a simultaneous extension of Theorem \ref{2p} and the result of Cabrelli and Gordillo \cite[Theorem 4.4]{CG} from integer to rational dilations.

\begin{theorem}\label{woj}
Let $A$ be an $n\times n$ expansive rational dilation and the lattice $\Ga= A\Zn+\Zn$. Let $N\in \N$ be such that $n \le 2(p-q)N $, where $p:=|\Ga/A\Z^n|$, and $q=|\Ga/\Z^n|$.

Then, for any $(A,\Gamma)$-MRA of multiplicity $N$ generated by $r$-regular scaling functions $\Phi=\{\vp^1, \ldots, \vp^N \}$,  there exists an orthonormal wavelet $\Psi=\{\psi^1,\ldots, \psi^{(p-q)N}\}$ consisting of $r$-regular functions.
\end{theorem}

\begin{proof}
Given a column vector consisting of functions $f_1,\ldots,f_j \in L^2(\R^n)$ and $g\in L^2(\R^n)$ we denote
\[
[\lan F,g \ran] = \begin{bmatrix} \lan f_1,g \ran \\ \vdots \\ \lan f_j, g \ran \end{bmatrix}.
\]

Let $\tilde \Phi$ to be the column vector as in \eqref{p0}. Since $E^\Gamma(\Phi)$ is an orthonormal basis of $V_0$ and $D_{A^{-1}}\tilde \Phi \in V_0$, we have
\begin{equation}\label{woj3}
D_{A^{-1}}\tilde \Phi = %\begin{bmatrix} 
\sum_{l=1}^N \sum_{k\in \Gamma} [\lan D_{A^{-1}}\tilde \Phi, T_k \vp^l \ran]  T_k \vp^l.
\end{equation}
Since scaling functions $\vp^l$ have faster than polynomial decay, for any $d>0$, there exists a constant $c_d>0$ such that
\begin{equation}\label{woj5}
|| [\lan D_{A^{-1}}\tilde \Phi, T_k \vp^l \ran] || \le c_d (1+|k|)^{-d} \qquad\text{for all }k \in \R^n.
\end{equation}
By  taking the Fourier transform of \eqref{woj3}, Lemma \ref{v0} yields $qN \times N$ matrix-valued function $M$ on $\Rn$ such that $M$ is $\Ga^*$-periodic and  
\[
b^{1/2}\widehat {\tilde \Phi}(B\xi) = M(\xi) \widehat \Phi(\xi)  \qquad\text{a.e. }\xi.
\]
The Fourier coefficients of $M$ are $qN \times N$ matrices indexed by $k\in \Gamma$ given by
\[
\begin{bmatrix} 
[\lan D_{A^{-1}}\tilde \Phi, T_k \vp^1 \ran] \ldots  [\lan D_{A^{-1}}\tilde \Phi, T_k \vp^N \ran]
\end{bmatrix}.
\]
By \eqref{woj5} the entries of $M$ are $C^\infty$ functions. Consequently, the corresponding polyphase matrix functions $M^{\uparrow d_i}$, $i=1,\ldots,p$, are also $C^\infty$ functions defined on $\T^n$.

By Theorem \ref{AK} $qN \times pN$ matrix function 
$\begin{bmatrix} 
M^{\uparrow d_1}(\xi) & M^{\uparrow d_2}(\xi) & \ldots & M^{\uparrow d_p}(\xi) \end{bmatrix}$
can be appended  by $(p-q)N \times pN$ matrix function
$\begin{bmatrix} 
H^{\uparrow d_1}(\xi) & H^{\uparrow d_2}(\xi) & \ldots &H^{\uparrow d_p}(\xi)
\end{bmatrix}$,
which is $C^\infty$ on $\T^n$, such that the resulting matrix \eqref{cri4} is unitary for a.e. $\xi$. Thus, the corresponding  $(p-q)N \times N$ matrix function $H$, for which $H^{\uparrow d_i}$'s are its polyphase decomposition, is a $C^\infty$ and $\Gamma^*$-periodic. Define $\Psi=\{\psi^1,\ldots, \psi^{(p-q)N}\} \subset L^2(\R^n)$ such that 
\eqref{p2} holds, i.e., $H$ is a high-pass filter of $\Psi$. By Corollary \ref{LH}, $\Psi$ is an orthonormal wavelet with respect to the dilation $A$ and the lattice $\Z^n$. Finally, by \cite[Lemma 5.13]{Woj}, wavelet family $\Psi$ is $r$-regular.
\end{proof}

Combining Theorems \ref{MRA} and \ref{woj} yields the existence of Meyer wavelets for rational dilations.

\begin{theorem}\label{main} Suppose $A$ is a rational $n\times n$ expansive dilation. Then, for some $L\in \N$, there exists an orthonormal wavelet $\Psi=\{\psi^1,\ldots, \psi^{L}\}$ in the Schwartz class. Moreover, each $\widehat \psi^i$ is a $C^\infty$ function with compact support.
\end{theorem}

\begin{proof}
Let $\Gamma=\Z^n+A\Z^n$.
By Theorem \ref{MRA}, for some $N\in \N$, there exists an $(A,\Gamma)$-MRA $\{V_j\}_{j\in\Z}$ of multiplicity $N$ such that its scaling functions $\{\vp^1, \ldots, \vp^N\} \subset V_0$ are in the Schwartz class. Moreover, each $\widehat \vp^i$ is a $C^\infty$ function with compact support. By the construction of an MRA in Theorem \ref{MRA}, we can choose $N$ to be arbitrary large so that $n \le 2(p-q)N $, where $p:=|\Ga/A\Z^n|$, and $q=|\Ga/\Z^n|$.
Hence, by Theorem \ref{woj} there exists an orthonormal wavelet $\Psi=\{\psi^1,\ldots, \psi^{L}\}$, $L=(p-q)N$, consisting of Schwartz class functions. By \eqref{p2}, each $\widehat \psi^i$ is a $C^\infty$ function with compact support.
\end{proof}

\section{Strictly expansive dilations} \label{S6}

In this section we generalize the results of Speegle and the author \cite{BS} on the existence  MRAs and Meyer wavelets associated with strictly expansive dilations with integer entries. We show that MRAs of multiplicity 1 exist for  strictly expansive dilations with real entries. This implies the existence of Meyer wavelets for rational dilations $A$, which are strictly expansive, with the smallest possible number of generators. We illustrate our results by presenting a natural lifting procedure of one dimensional rationally dilated wavelets to higher dimensions, which preserves the minimal number of generators.

\begin{definition} We say that an expansive $n\times n$ matrix dilation $A$
 is {\it strictly expansive} with respect to a full-rank lattice $\Ga \subset \R^n$ if there exists a compact 
 set $K\subset \R^n$  such that
\begin{align}
\label{sed1}
\sum_{k\in\Ga^*} \ch_K(\xi+k)=1 & \qquad\text{for a.e. }\xi\in\R^n,
\\
\label{sed2}
K  \subset B K^\circ & \qquad\text{where }B=A^\top \text{ and } K^\circ\text{ is the interior of }K.
\end{align}
\end{definition}

We shall prove the existence of $\infty$-regular MRAs associated with strictly expansive real dilations, which generalizes \cite[Theorem 3.2]{BS}. 

\begin{theorem}\label{se} Suppose $A$ is a $n\times n$ real dilation matrix and $\Ga \subset \R^n$ is a full rank lattice. 
If $A$ is strictly expansive with respect to $\Ga$, then there exists 
$(A,\Gamma)$-MRA of multiplicity $1$ such that its scaling function $\vp$ belongs to the Schwartz class. Moreover, $\widehat\vp$ is $C^\infty$ with compact support.
\end{theorem}

\begin{proof}
Suppose the compact set $K$ satisfies \eqref{se1} and \eqref{se2}.
Given $\ve>0$ define
\[
\begin{aligned}
K^{-\ve} &= \{\xi\in\R^n: \mathbf B(\xi,\ve) \subset K \} \\
K^{+ \ve} &= \{\xi\in\R^n: \mathbf B(\xi,\ve) \cap K \not= \emptyset\} .
\end{aligned}
\]
Note that $K^{-\ve}$ is closed, $K^{+\ve}$ is open,
and the interior of $K$ satisfies $K^\circ= \bigcup_{\ve>0} K^{-\ve}$.
Hence there exists $\ve>0$ such that
\bq\label{se1}
K^{+\ve} \subset B(K^{-\ve}).
\eq
Pick a $C^\infty$ function $g:\R^n \to [0,\infty)$ such that $\int_{\R^n} g=1$ and
\[
\supp g:= \{ \xi\in\R^n: g(\xi) \not= 0 \}  = \mathbf B(0,\ve).
\]
Let $f = \ch_K * g$ be a smoothing of $\ch_K$.
Clearly $f$ is in the class $C^\infty$, $0\le f(\xi) \le 1$, and
\bq\label{se2}
K^{-\ve} \subset \{\xi\in\R^n: f(\xi)=1 \}\subset \supp f \subset K^{+\ve}.
\eq
Moreover, by \eqref{se1}
\bq\label{se3}
\sum_{k\in\Ga^* } f(\xi+k)=
\sum_{k\in\Ga^*} T_k(\ch_K* g)(\xi)=  \bigg (\sum_{k\in\Ga^*} T_k \ch_K \bigg)*  g(\xi) =1 \qquad\text{for all }\xi \in\R^n.
\eq
Define function $\vp\in L^2(\R^n)$ by 
\[
\widehat \vp(\xi) = \sqrt{\vol(\Gamma)} f(\xi) \bigg( \sum_{k\in \Gamma^*} f(\xi+k)^2 \bigg)^{-1/2}.
\]
For every $\xi \in \Gamma^*$, only finitely many terms $f(\xi+k)$ are non-zero. Hence, the above sum defines a positive $C^\infty$ and $\Ga^*$-periodic function. Consequently, $\vp$ is a band-limited $C^\infty$ and hence in the Schwartz class. Since
\[
\sum_{k\in \Gamma^*} |\widehat \vp(\xi+k)|^2 = \vol(\Gamma) \quad\text{for all }\xi,
\]
the system of translates $E^\Ga(\vp)$ is an orthonormal basis of its closed linear span 
\[
V_0 = \ov{\spa} \{ T_\ga \vp: \ga \in \Ga  \}.
\]

Define the spaces $V_j$ by $V_j=D_{A^j}(V_0)$, $j\in\Z$. We claim that $(V_j)_{j\in \Z}$ is an MRA. The properties (M4) and (M5) hold automatically.  By \eqref{se2} and \eqref{se3}, if $\xi \in K^{-\ve}$, then $\widehat \vp (\xi+k) =0$ for all $k\in \Ga^* \setminus \{0\}$.
Hence, by Lemma \ref{L1}, we have
\bq\label{se5}
\check L^2(K^{-\ve}) =  \{ f \in L^2(\Rn): \supp \widehat f  \subset K^{-\ve} \}  \subset V_0.
\eq
On the other hand, by \eqref{se1} and \eqref{se2}
\bq\label{se6}
V_0 \subset \check L^2( \supp \vp) \subset \check L^2(K^{+\ve}) \subset   \check L^2(B(K^{-\ve})).
\eq
MRA properties (M1)--(M3) now follow easily. Indeed, by \eqref{se5} and \eqref{se6} we have for $j\in \Z$,
\[
V_j =  D_{A^j}(V_0) \subset D_{A^j}(\check L^2(B(K^{-\ve})))= 
\check L^2(B^{j+1}(K^{-\ve}))= D_{A^{j+1}}(\check L^2(K^{-\ve})) \subset V_{j+1} .
\]
Since $B$ is an expansive matrix we have $\bigcap_{j\in \Z} B^j(K^{-\ve}) = \{0\}$ and $\bigcup_{j\in \Z} B^j(K^{-\ve}) = \Rn$. This yields (M2) and (M3) and shows Theorem \ref{se}.
\end{proof}

\begin{remark}
When $A$ is an integer dilation matrix, Theorem \ref{se} can be alternatively deduced from the existence of a low-pass filter $m$ such that $m$ is $\Z^n$-periodic $C^\infty$ function satisfying 
\begin{equation}\label{se7}
\sum_{j=1}^p |m(\xi+\omega_p)|^2 = 1 \qquad\text{for all }\xi\in \R^n,
\end{equation}
where $B=A^\top$, and $\{\omega_1,\ldots,\omega_p\}$ is a transversal set of $\Z^n/A^\top\Z^n$. In addition, if $m$ is chosen such that $m(\xi)=1$ for $\xi$ in an appropriate neighborhood of the origin, then the function $\vp$, given by
\[
\widehat \vp(\xi) = \prod_{j=1}^\infty m ( (A^\top)^{-j} \xi ), \qquad\xi\in \R^n,
\]
is a scaling function of $(A,\Z^n)$-MRA, see \cite[Theorem 2.3 in Ch. 2]{B1} and \cite[Theorem 3.2]{BS}. A similar result for $r$-regular wavelets was shown in \cite[Theorem 2.3]{B2}. 

The construction of a low-pass filter $m$ follows the proof of Theorem \ref{se} with $\Gamma=\Z^n$. By \eqref{se3} we have a  compactly supported $C^\infty$ function $f:\R^n\to [0,1]$ satisfying $\sum_{k\in \Z^n} f(\xi+k)=1$ for all $\xi\in \R^n$ . Then, it is easy to verify that the function
\begin{equation}\label{se8}
m(\xi)= \sqrt{ \sum_{k\in \Z^n} f(A^\top(\xi+k)) }
\end{equation}
satisfies \eqref{se7}. However, care needs to be taken to guarantees that $m$ is actually a $C^\infty$ function.
Even if we assume that a non-negative $C^\infty$ function ``vanishes strongly" ($f(\xi)=0$ $\implies$ all its partial derivatives $\partial^\alpha f(\xi)=0$), then $\sqrt{f}$ might not be smooth. The following example gives two functions with smooth square roots such that their sum does not have a smooth square root
\[
g(t) = \begin{cases} \sin^2(1/t) e^{-1/t} + e^{-2/t} & t>0,
\\
0 &t \le 0.
\end{cases}
\]
To guarantee smoothness of the square root of a function, we can employ the following elementary lemma.

\begin{lemma}\label{sqrt}
Let $h: \R \to [0,\infty)$ be given by 
\[
h(t)= \begin{cases} e^{-1/t} &t>0, \\
0 & t\le 0.
\end{cases}
\]
Suppose $f_1,\ldots, f_k: \R^n \to [0,\infty)$ are $C^\infty$ functions. Then, $(h\circ f_1+\ldots+h \circ f_k)^{1/2}$ is a $C^\infty$ function.
\end{lemma}

Instead of defining function $m$ by \eqref{se8}, we let
\begin{equation}\label{se9}
m(\xi)=  \bigg(\sum_{k\in \Z^n} h(f(A^\top(\xi+k))) \bigg)^{1/2} 
\bigg( \sum_{k\in \Z^n} h(f(A^\top\xi+k)) \bigg)^{-1/2}.
\end{equation}
A simple calculation shows that $m$ satisfies \eqref{se7}. By Lemma \ref{sqrt} $m$ is a $C^\infty$ function. Since zero sets of both functions \eqref{se8} and \eqref{se9} are identical, \cite[Claim 3.3]{BS} holds and the rest of the proof of \cite[Theorem 3.2]{BS} remains unchanged.
\end{remark}

Combining Theorems \ref{woj} and \ref{sew} yields the existence of Meyer wavelets for rational dilations with the minimal number of generators.

\begin{theorem}\label{sew}
Suppose that $A$ is $n\times n$ matrix with rational entries such that:
\begin{enumerate}[(i)]
\item $A$ is strictly expansive with respect to the lattice $\Ga= A\Z^n+\Z^n$, 
\item we have $n \le 2(p-q)$, where $p=|\Ga/A\Zn|$ and $q=|\Ga/\Z^n|$.
\end{enumerate}
Then, there exists an orthonormal wavelet $\Psi$ with respect to the dilation $A$ and the lattice $\Z^n$, which consists of $(p-q)$ functions in the Schwartz class. Moreover, the Fourier transform of each function in $\Psi$ is $C^\infty$ with compact support.
\end{theorem}

\begin{proof} It suffices to follow the proof of Theorem \ref{main} by applying Theorem \ref{sew} instead of Theorem \ref{MRA}.
\end{proof}

Theorem \ref{sew} yields the existence of rationally dilated Meyer wavelets on the real line, which were constructed by Auscher \cite{Au0, Au1}. The construction of such wavelets for the dilation factor $\frac{q+1}{q}$ is attributed to G. David, see \cite[Section 3.3 in Part II]{KLR}. The construction of wavelets for the dilation factor $3/2$ is nicely explained in the book of Daubechies \cite[Section 10.4]{Dau}.

\begin{example}\label{ex1}
Let $p, q\in \N$ be relatively prime such that $p>q$. Let $A=p/q$ be a dilation factor. By Theorem \ref{se} there exists a $(A,\frac{1}{q}\Z)$-MRA with scaling function $\vp$ such that $\widehat \vp$ is $C^\infty$ with compact support. By \eqref{p1} there exists $\tfrac1q \Z$-periodic $C^\infty$ functions $m_j$, $j=1,\ldots,q$, such that
\begin{equation}\label{ex2}
(\tfrac{p}{q})^{1/2} e^{-2\pi i j p\xi /q^2} \widehat\vp (\tfrac{p}{q}\xi ) = m_j(\xi) \widehat \vp(\xi) \qquad \text{for } \xi \in \R.
\end{equation}
The vector-valued function $M=\begin{bmatrix} m_1 \ldots m_q \end{bmatrix}^\top$ is a low-pass filter of $\vp$. By Lemma \ref{L3}, the following $q\times p$ matrix has orthogonal rows for all $\xi \in \R$,
\[
\frac{1}{\sqrt p}
\begin{bmatrix} 
m_1(\xi+\omega_1) & m_1(\xi+\omega_2) & \ldots & m_1(\xi+\omega_p)
\\
\vdots & \vdots & & \vdots \\
\\
m_q(\xi+\omega_1) & m_q(\xi+\omega_2) & \ldots & m_q(\xi+\omega_p)
\end{bmatrix}.
\]
By Theorem \ref{sew} there exists an orthonormal wavelet $\Psi=\{\psi^1,\ldots,\psi^{p-q} \}$ with respect to the dilation $A$ and the lattice $\Z$ such that each $\widehat \psi^i$ is $C^\infty$ with compact support. By \eqref{p2} there exists $\tfrac1q \Z$-periodic $C^\infty$ functions $h_j$, $j=1,\ldots,p-q$, such that
\begin{equation}\label{ex3}
(\tfrac{p}{q})^{1/2} \widehat\psi^j (\tfrac{p}{q}\xi ) = h_j(\xi) \widehat \vp(\xi) \qquad \text{for } \xi \in \R.
\end{equation}
The vector-valued function $H=\begin{bmatrix} h_1 \ldots h_{p-q} \end{bmatrix}^\top$ is a high-pass filter of $\Psi$.  By Theorem \ref{cri}, the following $p\times p$ matrix is orthogonal for all $\xi \in \R$,
\[
\frac{1}{\sqrt p}
\begin{bmatrix} 
m_1(\xi+\omega_1) & m_1(\xi+\omega_2) & \ldots & m_1(\xi+\omega_p)
\\
\vdots & \vdots & & \vdots \\
\\
m_q(\xi+\omega_1) & m_q(\xi+\omega_2) & \ldots & m_q(\xi+\omega_p)
\\
h_1(\xi+\omega_1) & h_1(\xi+\omega_2) & \ldots & h_1(\xi+\omega_p)
\\
\vdots & \vdots & & \vdots \\
\\
h_{p-q}(\xi+\omega_1) & h_{p-q}(\xi+\omega_2) & \ldots & h_{p-q}(\xi+\omega_p)
\end{bmatrix}.
\]
\end{example}

The next example shows how to lift one dimensional wavelet basis to higher dimensions. This is typically done by tensoring for dyadic dilations $A=2 \mathbf I_n$  as explained in several books on wavelets \cite{Dau, Mey, NPS, Woj}. However, we shall present a variant of this construction for a less obvious choice of a dilation matrix.

\begin{example}\label{ex4}
Let $n\ge 2$. Keeping the same notation as in Example \ref{ex1}, we let $\tilde A$ be $n\times n$ matrix given by
\[
\tilde A= \begin{bmatrix}
 & \frac{1}{q} & & \\
 & & 1 & & \\
 & & & \ddots & \\
 & & & & 1 \\
p & & & &
\end{bmatrix}
\]
When $n=2$, the matrix $\tilde A$ has zero diagonal and off-diagonal terms $p$ and $1/q$.
Define the lattice $\Gamma = \Z^n+\tilde A\Z^n = \frac{1}{q}\Z \times \Z^{n-1} $. Let $\vp$ be the scaling function from Example \ref{ex1}. Define the function $\phi$ on $\R^n$ by
\[
\phi(x_1,\ldots,x_n) =  q^{-(n-1)/2} \vp(x_1)  \vp(\tfrac1q x_{2})  \ldots \vp(\tfrac1q x_n) \qquad (x_1,\ldots,x_n) \in \R^n.
\]
Since
\[
\widehat \phi(\xi_1,\ldots,\xi_n)= q^{(n-1)/2} \widehat\vp(\xi_1) \widehat\vp( q \xi_{2}) \ldots \widehat\vp(q \xi_n) \qquad (\xi_1,\ldots,\xi_n) \in \R^n,.
\]
we have
\begin{multline*}
\sum_{k\in \Gamma^*} |\widehat\phi(\xi+k)|^2 
= q^{n-1} \sum_{k_1\in q\Z} |\widehat\vp(\xi_1+k_1)|^2  \sum_{k_{2} \in \Z} |\widehat\vp(q (\xi_{2}+k_{2}))|^2  \ldots 
\sum_{k_n\in \Z} |\widehat\vp(q(\xi_n+k_n))|^2 \\
= \frac{1}{q} \qquad \text{for a.e. } \xi=(\xi_1,\ldots,\xi_n) \in \R^n.
\end{multline*}
Hence, the system of translates $E^\Gamma(\phi)$ is an orthonormal basis of its closed linear span $\tilde V_0$. Let $\mathbf e=(1,0,\ldots,0) \in \R^q$. For $j=1,\ldots,q$ by \eqref{ex2} we have 
\[\begin{aligned}
\widehat{ (D_{\tilde A^{-1}} T_{j/q \mathbf e}\phi) } &(\xi_1,\ldots,\xi_n) = D_{\tilde A^\top} \widehat {T_{j/q \mathbf e}\phi} (\xi_1,\ldots,\xi_n)  
= (\tfrac{p}{q})^{1/2}  e^{-2\pi i j p\xi_n /q}  \widehat \phi( p \xi_n, \tfrac{1}q \xi_1, \xi_2,\ldots, \xi_{n-1})
\\
&=p^{1/2}q^{(n-2)/2} \widehat \vp(\xi_1)  \widehat \vp(q \xi_{2}) \ldots \widehat\vp( q\xi_{n-1})  e^{-2\pi i j p\xi_n /q} \widehat\vp(p\xi_n)
\\
&=q^{(n-1)/2} \widehat \vp(\xi_1)  \widehat \vp(q \xi_{2}) \ldots \widehat\vp( q\xi_{n-1})  m_j(q\xi_n) \widehat\vp(q\xi_n)= m_j(q\xi_n) \widehat \phi(\xi_1,\ldots,\xi_n).
\end{aligned}
\]
This calculation and Lemma \ref{v0} imply that the space $\tilde V_0$ is refinable, i.e., $\tilde V_0 \subset \tilde V_{1} :=D_{\tilde A}\tilde V_0$. Hence, $\phi$ is a scaling function of $(\tilde A,\Gamma)$-MRA $\{\tilde V_j\}_{j\in \Z}$. Moreover, the low-pass filter of $\phi$ is the function $(\xi_1,\ldots,\xi_n) \mapsto M(q\xi_n)$. Hence, it is nearly identical as the low-pass filter $M$ of $\vp$. 

Finally, define functions $\theta^j$,  $ j=1,\ldots,p-q$ on $\R^n$ by
\[
\theta^j(x_1,\ldots,x_n) =  q^{-(n-1)/2} \psi^j(x_1)  \vp(\tfrac1q x_{2})  \ldots \vp(\tfrac1q x_n) \qquad (x_1,\ldots,x_n) \in \R^n.
\]
A similar calculation as above  using \eqref{ex3} shows that
\[
\begin{aligned}
\widehat{ (D_{\tilde A^{-1}}  \theta^j) } &(\xi_1,\ldots,\xi_n) = (\tfrac{p}{q})^{1/2}   \widehat \theta^j( p \xi_n, \tfrac{1}q \xi_1, \xi_2,\ldots, \xi_{n-1})
\\
&=p^{1/2}q^{(n-2)/2} \widehat \vp(\xi_1)  \widehat \vp(q \xi_{2}) \ldots \widehat\vp( q\xi_{n-1})  \widehat \psi^j(p\xi_n)
\\
&=q^{(n-1)/2} \widehat \vp(\xi_1)  \widehat \vp(q \xi_{2}) \ldots \widehat\vp( q\xi_{n-1})  m_j(\xi_n) \widehat\vp(q\xi_n)= h_j(q\xi_n) \widehat \phi(\xi_1,\ldots,\xi_n).
\end{aligned}
\]
This calculation shows that $\Theta=\{\theta^1,\ldots,\theta^{p-q} \}$ is a an orthonormal wavelet for the dilation $\tilde A$ and the lattice $\Z^n$ with the high-pass filter  $(\xi_1,\ldots,\xi_n) \mapsto H(q\xi_n)$, where $H$ is the high-pass filter of $\Psi$. 
\end{example}

Consequently, the dilation $\tilde A$ as in Example \eqref{ex4} yields a natural lifting of one dimensional rationally dilated wavelets to higher dimensions, while keeping nearly identical low-pass and high-pass filters, and thus maintaining the same size of a wavelet family. This is in contrast to the dilation $A=\frac{p}{q} \mathbf I_n$, which results in a wavelet family of size $p^n-q^n$. We leave details of this construction to the reader.

While Theorem \ref{sew} yields the existence of Meyer wavelets with minimal number of generators, it requires a strict expansiveness assumption on a rational dilation $A$. The result of Speegle and the author \cite{BS} shows that Meyer wavelets with minimal ($|\det A|-1$) generators exist for all integer dilations $A$ in $\R^2$, despite that some dilations $A$ might not be strictly expansive. On the other hand, Theorem \ref{main} guarantees the existence of Meyer wavelets for any rational dilation albeit with possibly large number of generators. It is an open problem to characterize dilations $A$ for which there exist Meyer wavelets with a minimal number of generators. This problem remains open already in dimension $3$ for integer dilations $A$.

\section{Limitations on the existence of well-localized wavelets} \label{S7}

In this section we show that well-localized wavelets associated with an MRA can only exist for rational dilations, thus extending the one dimensional result of Lemari\'e-Rieusset  \cite{KLR}. To achieve this we shall use results on shift-invariant operators with localized kernels, which were developed by Lemari\'e-Rieusset  \cite{KLR, Lem2}.

Recall that a positive function $\omega$ on $\R^n$ is a Beurling weight if there exists constants $C,M>0$ such that:
\begin{enumerate}[(i)]
\item for all $x\in \R^n$, $1/C \le \omega(x) \le C (1+|x|)^M$,
\item $\int_{\R^n} \frac1{\omega(x)} dx <\infty$,
\item for all $x\in \R^n$, $\int_{\R^n} \frac 1{\omega(x-y)\omega(y)} dy \le C \frac 1{\omega(x)}$, and
\item for all $x,y\in \R^n$, $\omega(x+y) \le C\omega(x)\omega(y)$.
\end{enumerate}
For our purposes we shall only need to consider Beurling weights of the form $\omega(x)=(1+\rho_A(x))^\eta$, $\eta>1$. Here, $\rho_A$ is a quasi-norm associated with a dilation $A$, which satisfies three conditions:
\begin{enumerate}[(a)]
\item $\rho_A$ is $C^\infty$ on $\R^n \setminus \{0\}$ and continuous at $0$,
\item for all $x\ne 0$, $\rho_A(x)=\rho_A(-x)>0$, and
\item for all $x\in \R^n$, $\rho_A(x)=|\det A| \rho_A(x)$.
\end{enumerate}
The smoothness assumption (a) is not essential as it can be replaced by the triangle inequality:
\begin{enumerate}[(d)]
\item there exists $c>0$ such that for all $x,y\in \R^n$, $\rho_A(x+y)\le c (\rho_A(x)+\rho_A(y))$.
\end{enumerate}
It is well-known that quasi-norms associated with the same dilation matrix $A$ are equivalent, see \cite[Lemme 17]{Lem2} and \cite[Lemma 2.4]{B1}. By \cite[Lemma 22]{Lem2}, for $\eta>1$, $\omega(x)=(1+\rho_A(x))^\eta$ is a Beurling weight on $\R^n$. Let $L^2(\omega)$ denote the weighted Lebesgue space consisting of measurable $f$ satisfying 
\[
\int_{\R^n} |f(x)|^2 \omega(x) \,dx<\infty.
\]

\begin{definition}
Let $P_0: L^2(\R^n) \to L^2(\R^n)$ be a bounded linear operator. 
We say that:
\begin{enumerate}[(i)]
\item
$P_0$ is shift-invariant (SI) if
\[
T_kP_0 = P_0 T_k \qquad\text{for all }k \in\Z^n,
\]
\item
$P_0$ is an integral operator if there exists $p \in L^1_{loc}(\R^n \times \R^n)$ such that
\[
\lan P_0 d, g \ran = \int_{R^n} \int_{\R^n} p(x,y) f(y) \ov{g(x)} \, dy \, dx \qquad\text{for all } f,g \in C_c^\infty (\R^n),
\]
\item
$P_0$ has a localized kernel with respect to the weight $\omega(x)=(1+\rho_A(x))^\eta$, $\eta>1$, if
\[
 \int_{[0,1]^n} \int_{\R^n}  \omega(x-y)( |p(x,y)|^2 + |p(y,x)|^2) \, dy \, dx. 
\]
\end{enumerate} 
\end{definition}

We shall employ the invariant projection theorem due to Lemari\'e-Rieusset, which can be found in \cite[Theorem 1 in Chapter 3, Part II]{KLR} and \cite[Th\'eor\`eme 1]{Lem2}. While Theorem \ref{ipt} holds for general symmetric Beurling weights $\omega(x)=\omega(-x)$, we will only need it for the weights of the form $\omega(x)=(1+\rho_A(x))^\eta$, where $\eta>1$.

\begin{theorem} \label{ipt}
Let $P_0: L^2(\R^n) \to L^2(\R^n)$ be a bounded projection, $P_0 \circ P_0=P_0$, which is SI.  Let $V_0 = \operatorname{Im} P_0$ and $V_0^*= (\operatorname{Ker} P_0)^\perp$. 
Let $\omega$ be a symmetric Beurling weight. Then the following are equivalent:
\begin{enumerate}[(i)]
\item $P_0$ is an integral operator with $\omega$-localized kernel,
\item $V_0$ has a Riesz basis $E^{\Z^n}(\vp^1,\ldots,\vp^N)$ with each $\vp^j \in L^2(\omega)$ and $V_0^*$ has a Riesz basis $E^{\Z^n}( \vp^1_*,\ldots,\vp^N_*)$ with each $\vp^j_* \in L^2(\omega)$ that satisfy biorthogonality relation:
\[
\lan T_k \vp^j, T_{k'} \vp^{j'}_* \ran = \delta_{k,k'} \delta_{j,j'} 
\qquad\text{for } k,k' \in \Z^n, \ j,j'=1,\ldots,N,
\]
\end{enumerate}
\end{theorem}

Furthermore, Lemari\'e-Rieusset showed the following result on the existence of MRAs associated with wavelet bases \cite[Th\'eor\`eme 3]{Lem2}.

\begin{theorem}\label{lr}
 Let $A$ be $n\times n$ real expansive matrix. Suppose that $\{\psi^1,\ldots, \psi^L\} \subset L^2(\R^n)$ and $\{\psi^1_*,\ldots, \psi^L_*\} \subset L^2(\R^n)$ is a pair of biorthogonal Riesz wavelets. That is, wavelet systems
\[
\{ D_{A^j} T_k \psi^l: j\in\Z, k\in \Z^n, l=1,\ldots,L\}
\]
and
\[
\{ D_{A^j} T_k \psi_*^l: j\in\Z, k\in \Z^n, l=1,\ldots,L\}
\]
are Riesz bases in $L^2(\R^n)$ satisfying biorthogonality relation
\[
\lan D_{A^j} T_k \psi^l, D_{A^{j'}} T_{k'} \psi^{l'}_* \ran =\delta_{j,j'} \delta_{k,k'} \delta_{l,l'}
\qquad
j, j'\in\Z, \ k, k' \in \Z^n,\  l,l'=1,\ldots,L.
\]
Assume that:
\begin{enumerate}[(i)]
\item there exists $\eta>1$ such that each $\psi^l$ and each $\psi^l_*$ belongs to $L^2((1+\rho_A)^\eta)$, and
\item the spaces of negatives dilates
\begin{equation}\label{lr1}
V_0= \ov{\spa}\{ D_{A^j} T_k \psi^l: j<0, k\in \Z^n, l=1,\ldots,L\}
\end{equation}
and
\begin{equation}\label{lr2}
V^*_0= \ov{\spa}\{ D_{A^j} T_k \psi^l_*: j<0, k\in \Z^n, l=1,\ldots,L\}
\end{equation}
are shift-invariant.
\end{enumerate}

Then, the operator
\begin{equation}\label{lr3}
P_0f = \sum_{l=1}^L \sum_{j<0} \sum_{k\in \Z^n} \lan f, D_{A^j} T_k \psi^l_* \ran D_{A^j} T_k \psi^l
\end{equation}
is a shift-invariant projection and $P_0$ is an integral operator with $(1+\rho_A)^\eta$-localized kernel.
\end{theorem}

\begin{remark}
Note that the formulation of Theorem \ref{lr} differs slightly from \cite[Th\'eor\`eme 3]{Lem2} since we do not assume a priori that the dilation $A$ has integer entries, $A\Z^n\subset \Z^n$. For such dilations the spaces of negative dilates $V_0$ and $V_0^*$ are automatically SI. Instead, we make the assumption that $V_0$ and $V_0^*$ are both SI, as it is done in one dimensional variant of Theorem \ref{lr} in the book of Kahane and Lemari\'e-Rieusset \cite[Theorem 2 in Chapter 3, Part II]{KLR}. In addition, the assumption that there exists $\alpha>0$ such that each $\widehat{\psi^l}$ and $\widehat{\psi^l_*}$ belongs to $L^2((1+\rho_{A^\top})^\alpha)$ in \cite{Lem2} can be omitted (as it is made to satisfy hypotheses of the lemma of ``vaguelettes"). Moreover, since we do not assume that $A\Z^n\subset \Z^n$, we can not conclude that there is a relationship between the number $N$ of generators of the space $V_0$ and size $L$ of the wavelet, which takes the form $L=N(|\det A|-1)$. A similar result was also shown by Auscher \cite{Au3}.
\end{remark}

We are now ready to formulate the main result of this section, which is a higher dimensional generalization a result due to Lemari\'e-Rieusset  \cite[Theorem 2 in Chapter 3, Part II]{KLR}.

\begin{theorem}\label{mb}
Let $A$ be $n\times n$ real expansive matrix. Suppose that $\{\psi^1,\ldots, \psi^L\} \subset L^2(\R^n)$ and $\{\psi^1_*,\ldots, \psi^L_*\} \subset L^2(\R^n)$ is a pair of biorthogonal Riesz wavelets satisfying the assumptions of Theorem \ref{lr}. In particular, this assumption is met when there exists $(A,\Z^n)$-orthonormal wavelet $\{\psi^1,\ldots, \psi^L\} \subset L^2(\R^n)$ such that for some $\eta>1$ we have $\psi^l \in L^2((1+\rho_A)^\eta)$ for all $l=1,\ldots,L$, and the space of negative dilates \eqref{lr1} is SI. Then, $A$ is a rational matrix.
\end{theorem}

To prove Theorem \ref{mb} we need to employ a series of lemmas involving function spaces associated with Beurling weights \cite{Lem2}.

\begin{definition}
For a Beurling weight $\omega$, we define an anisotropic Sobolev space $H_\omega$ as
\[
H_\omega= \{ f\in L^2(\R^n): \exists g \in L^2(\omega) \ f=\widehat g \},
\qquad
||f||_{H_\omega} = ||g||_{L^2(\omega)}.
\]
Define the periodic analogue of the space $H_\omega$ by
\[
K_{\omega} =\bigg \{m\in L^2(\T^n): m = \sum_{k\in \Z^n} a_k e^{-2\pi i \lan k, \xi \ran}, ||m||_{K_\omega} = \bigg(\sum_{k\in \Z^n} |a_k|^2 \omega(k) \bigg)^{1/2}<\infty \bigg\}.
\]
\end{definition}

Observe that $H_\omega \subset C_0(\R^n)$ and
\[
||f||_\infty \le C ||f||_{H_\omega} \qquad\text{for }f\in H_\omega.
\]
Likewise, $K_\omega \subset C(\T^n)$ and 
\[
||m||_\infty \le C ||m||_{K_\omega} \qquad\text{for }m \in K_\omega.
\]
We will use two facts about $H_\omega$ and $K_\omega$ from \cite[Proposition 5]{Lem2}. First, the space $K_\omega$ is an inverse closed subalgebra of the Wiener algebra 
\[
A(\T^n)= \bigg \{m\in L^2(\T^n): m = \sum_{k\in \Z^n} a_k e^{-2\pi i \lan k, \xi \ran}, ||m||_{A(\T^n)} = \sum_{k\in \Z^n} |a_k|<\infty \bigg\}
\]
with pointwise multiplication. Second, functions in $m\in K_\omega$ define multipliers in $H_\omega$. That is, there exists a constant $C>0$ such that
\begin{equation}\label{muh}
|| m f||_{H_\omega} \le C ||m||_{H_\omega} ||f||_{H_\omega}.
\end{equation}
\

We need the following crucial lemma due to Coifman and Meyer \cite[Lemme 2]{CM1}.

\begin{lemma}\label{CM}
Let $K\subset \R^n$ be a compact set satisfying $(K-K) \cap \Z^n =\{0\}$. For a Beurling weight $\omega$, there exists constants $C_1, C_2>0$, such that for any sequence $(f_k)_{k\in \Z^n}$ of $C^\infty$ functions supported in $K$, we have
\begin{equation}\label{CM0}
C_1 \| F \|_{H_\omega} \le \bigg( \sum_{k\in\Z^n} ||f_k||^2_{H_\omega} \bigg)^{1/2} \le C_2 ||F||_{H_\omega}, 
\qquad\text{where }F(x)=\sum_{k\in \Z^n} f_k(x-k).
\end{equation}.
\end{lemma}

As a consequence of Lemma \ref{CM} we deduce the following result, see also \cite[Lemma 6 in Section 2.4]{Mey}.

\begin{lemma}\label{cm}
Let $f\in H_\omega$ and $g\in C^\infty(\R^n)$ be compactly supported. Then,
\begin{equation}\label{cm1}
\sum_{k\in \Z^n} ||f T_k g ||_{H_\omega}^2 \le C ||f||_{H_\omega}.
\end{equation}
\end{lemma}

\begin{proof}
Since $g$ is compactly supported, there exists a finite collection $(g_j)_{j=1}^N$ of compactly supported $C^\infty$ functions such that
\begin{align*}
\sum_{j=1}^N g_j &= g
\\
(K_j-K_j) \cap \Z^n &=\{0\} \qquad\text{where } K_j = \supp g_j.
\end{align*}
Fix $j=1,\ldots, N$. Applying \eqref{muh} and Lemma \ref{CM} for functions $f_k = T_{-k}f g_j$ yields
\[
\bigg(\sum_{k\in \Z^n} ||f T_k g_j ||_{H_\omega}^2\bigg)^{1/2}
= \bigg(\sum_{k\in \Z^n} ||T_{-k} f g_j ||_{H_\omega}^2 \bigg)^{1/2} \le C_2 \| f  m \|_{H_\omega}
\le C_2 ||m||_{K_\omega} ||f||_{H_\omega},
\]
where $m= \sum_{k\in\Z^n} T_k g_j$ is a $C^\infty$ and $\Z^n$-periodic function. Hence, $m\in K_\omega$ and we obtain \eqref{cm1} for $g_j$ and some constant $C_j$. This yields the required bound for $g$ since
\[
\sum_{k\in \Z^n} ||f T_k g ||_{H_\omega}^2 \le \sum_{k\in \Z^n} \bigg( \sum_{j=1}^N ||f T_k g_j ||_{H_\omega} \bigg)^2
\le \sum_{k\in \Z^n}  \sum_{j=1}^N N ||f T_k g_j ||^2_{H_\omega} \le \sum_{j=1}^N C_jN  ||f||_{H_\omega}.
\qedhere
\]
\end{proof}

We also need the following lemma about extra invariant SI spaces. For the study of SI spaces in $L^2(\R^n)$ with extra invariance, see \cite{ACP}.

\begin{lemma} \label{esi}
Let $V_0 \subset L^2(\R^n)$ be a finitely generated $\Z^n$-SI space. Let $\Phi \subset L^2(\R^n)$ be such that $E^{\Z^n}(\Phi)$ is a Parseval frame of $V_0$. Define 
\begin{equation}\label{esi1}
\Lambda = \{y \in \R^n: T_{ty}(V_0)  \subset V_0 \quad\text{for all }t\in \R \}.
\end{equation}
Then,
\begin{equation}\label{esi2}
\sum_{\vp\in\Phi} \sum_{k\in \Z^n \cap \Lambda^\perp} |\widehat \vp(\xi+k)|^2 \in \N\cup\{0\} \qquad\text{for a.e. }\xi \in \R^n.
\end{equation}
\end{lemma}

\begin{proof}
If $\Lambda=\{0\}$, then the expression in \eqref{esi2} represents the dimension function $\dim_{V_0}$ of the space $V_0$ by \cite[Lemma 2.3]{BR1}. Since $V_0$ is finitely generated, we have $\dim_{V_0}(\xi) \in \N\cup\{0\}$ for a.e. $\xi$. Note that $\Lambda$ is the largest subspace of $\R^n$ such that the space $V_0$ is invariant under translates in $\Lambda$. By the characterization of closed subgroups of $\R^n$ containing $\Z^n$, see \cite[Theorem 3.9]{ACP}, the set $\Lambda' = \Z^n \cap \Lambda$ is a lattice of rank $r=\dim \Lambda$. 

For $j\in \N_0$, consider a full rank lattice $\Gamma_j=\Z^n+2^{-j} \Lambda'$. The quotient group $\Gamma_j/\Z^n$ is isomorphic with the group $(2^{-j}\Lambda')/\Lambda'$, which has order $2^{rj}$. Since $\Gamma_0=\Z^n$, we have
\[
\vol(\Gamma_j) = 2^{-rj} \vol(\Gamma_0) = 2^{-rj}.
\]
For any $y\in 2^{-j}\Lambda'$, the collection $E^{y+\Z^n}(\Phi)= E^{\Z^n}(T_y \Phi)$ is a Parseval frame of $T_y(V_0)=V_0$. Hence, $E^{\Gamma_j}(\Phi)$ is a tight frame with constant $2^{jr}$ of the space $V_0$. Equivalently, $2^{-rj/2} E^{\Gamma_j}(\Phi)$ is a Parseval frame of $V_0$. Let $\sigma_{V_0}$ be the spectral function of the space $V_0$, see \cite{BR1, BR2}. By \cite[(2.3) and Corollary 2.7]{BR2}, the dimension function of $\Ga_j$-SI space $V_0$ satisfies
\begin{equation}\label{esi3}
\begin{aligned}
\dim_{V_0}^{\Gamma_j}(\xi) = \sum_{k\in \Gamma_j^*} \sigma_{V_0}(\xi+k) 
& = \frac{2^{-rj}}{\vol(\Gamma_j)} \sum_{\vp\in \Phi} \sum_{k\in \Gamma_j^*} |\widehat \vp(\xi+k)|^2
\\
& = \sum_{\vp\in \Phi} \sum_{k\in \Gamma_j^*} |\widehat \vp(\xi+k)|^2 \in \N \cup\{0\} \qquad\text{for a.e. }\xi.
\end{aligned}
\end{equation}
Since $\Gamma_j \subset \Gamma_{j+1}$, we have $\Gamma_{j+1}^* \subset \Gamma_j^*$ for all $j\in \N$. 
Moreover, $\overline{\bigcup_{j\in \N} \Gamma_j }= \Z^n + \Lambda$ implies that $\bigcap_{j\in \N} \Gamma_j^* = \Z^n \cap \Lambda^\perp$. Thus, by \eqref{esi3}, $(\dim_{V_0}^{\Gamma_j})_{\in \N}$ is a monotone sequence converging to the expression in \eqref{esi2}, which also has values in $\N \cup\{0\}$.
\end{proof}

We need an extension of the extra invariance lemma \cite[Lemma 6 in Chapter 3, Part II]{KLR} to higher dimensions. While Lemma \ref{irr} holds for general Beurling weights as in \cite{Lem2}, we shall need it only for weights of the form $\omega=(1+\rho_A)^\eta$, $\eta>1$.

\begin{lemma}\label{irr}
Let $V_0 \subset L^2(\R^n)$ be a closed linear subspace such that:
\begin{enumerate}[(i)]
\item
$V_0$ has a Riesz basis $E^{\Z^n}(\vp^1,\ldots,\vp^N)$ with each $\vp^j \in L^2(\omega)$, where ,
\item
$V_0$ is shift invariant under some lattice $\Gamma \subset \R^n$, i.e., $T_k(V_0) \subset V_0$ for $k\in \Gamma$.
\end{enumerate}
Then, $\Gamma$ is a rational lattice. That is, there exists $n\times n$ invertible rational matrix $P$ such that $\Gamma=P\Z^n$.
\end{lemma}

\begin{proof}
By \cite[Lemme 16]{Lem2}, the assumption (i) implies the existence of an orthonormal basis $E^{\Z^n}(\psi^1,\ldots,\psi^N)$ of $V_0$ with each $\psi^j \in L^2(\omega)$. 
Indeed, this follows from the fact the $N\times N$ auto-correlation matrix $M(\xi)$, which is given in two equivalent forms by
\[
\begin{aligned}
M(\xi) & = \begin{bmatrix} \sum_{k\in \Z^n} \lan \vp^l, T_k \vp^j \ran e^{-2\pi i \lan k, \xi \ran} \end{bmatrix}_{l,j=1,\ldots,N}
\\
 & = \begin{bmatrix} \sum_{k\in \Z^n}  \widehat{\vp^l}(\xi+k)  \ov{ \widehat{\vp^j}(\xi+k)} \end{bmatrix}_{l,j=1,\ldots,N}
\qquad\text{for } \xi \in \T^n=\R^n/\Z^n,
\end{aligned}
\]
has entries in the algebra $K_\omega$. Since $M(\xi)$ is invertible for all $\xi\in \T^n$ and $K_\omega$ is an inverse closed algebra, the matrix function $N(\xi)=M(\xi)^{-1/2}$ also has entries in $K_\omega$ as a consequence of the following matrix identity for positive definite matrices $M$
\[
M^{-1/2} = \frac{2}{\pi} \int_0^\infty (\mathbf I + t^2 M)^{-1} dt.
\]
Define a generator $\psi^i$ by 
\[
\widehat{\psi^i}(\xi) = \sum_{j=1}^n N_{i,j}(\xi) \widehat{\vp^{j}(\xi)}.
\]
Then, it follows that $E^{\Z^n}(\psi^1,\ldots,\psi^N)$ is an orthonormal basis of $V_0$ with each $\psi^i \in L^2(\omega)$.
Therefore, without loss of generality we can assume that $E^{\Z^n}(\vp^1,\ldots,\vp^N)$ is an orthonormal basis of $V_0$.  This implies that
\begin{equation}\label{irr4}
 \sum_{i=1}^N \sum_{k\in \Z^n} |\widehat{\vp^i}(\xi+k)|^2 =N \qquad\text{for }\xi\in \R^n.
\end{equation}

On the contrary, suppose that $V_0$ is invariant under shifts of a lattice $\Gamma$, which is not rational. Thus, the closure of $\Z^n+\Gamma$ contains a non-trivial subspace. Let $\Lambda$ be the maximal subspace contained in $\ov{\Z^n+\Gamma}$. Thus, $V_0$ is closed under translates in $\Lambda$, which is given by \eqref{esi1}. By Lemma \ref{esi} 
\begin{equation}\label{irr5}
\sum_{j=1}^N \sum_{k\in \Z^n \cap \Lambda^\perp} |\widehat {\vp^j}(\xi+k)|^2 \in \N\cup\{0\} \qquad\text{for a.e. }\xi \in \R^n.
\end{equation}
Let $K$ be any compact subset of  $ \R^n$. Take $g\in C^\infty$ with compact support such that $g \equiv 1$ on $K$.
Hence, for all $\xi\in K$ we have
\[
|\widehat {\vp^j}(\xi+k)|^2 \le  ||T_{-k} (\widehat {\vp^j}) g||_\infty^2
=  ||\widehat {\vp^j} T_k g||_\infty^2.
\]
Since $\widehat{\vp^j} \in H_\omega$, by Lemma \ref{cm} we have
\[
\sum_{k\in \Z^n \cap \Lambda^\perp} |\widehat {\vp^j}(\xi+k)|^2
\le 
 \sum_{k\in \Z^n \cap \Lambda^\perp}  ||\widehat {\vp^j} T_k g||_\infty^2
\le \sum_{k\in \Z^n}  ||\widehat {\vp^j} T_k g||_\infty^2 \le C || \widehat {\vp^j} ||_\infty^2  <\infty.
\]
Since functions $\widehat{\vp^j}$ are continuous, this implies that the series in \eqref{irr5} converges uniformly on compact sets to a continuous function. Consequently, for some $m\in \N$ we have
\begin{equation}\label{irr6}
\sum_{j=1}^N \sum_{k\in \Z^n \cap \Lambda^\perp} |\widehat {\vp^j}(\xi+k)|^2 = m \qquad\text{for all }\xi \in \R^n.
\end{equation}
This contradicts \eqref{irr4} since $\Lambda^\perp$ is a proper subspace of $\R^n$ and hence the quotient group $\Z^n/(\Z^n \cap \Lambda^\perp)$ is infinite. Therefore, the lattice $\Gamma$ must be rational.
\end{proof}

Finally, we need a simple variant of a lemma due to Hoover and the author \cite[Lemma 4.1]{BH}.

\begin{lemma}\label{exp}
Suppose that $\Psi=\{\psi^1,\ldots, \psi^L\} \subset L^2(\R^n)$ is an  $(A,\Z^n)$-Riesz wavelet 
such that the spaces of negatives dilates
\[
V_0= \ov{\spa}\{ D_{A^j} T_k \psi^l: j<0, k\in \Z^n, l=1,\ldots,L\}
\]
is $\Z^n$-SI. Then, $V_0$ is $\Ga$-SI, where $\Ga=A\Z^n + \Z^n$.
\end{lemma}

\begin{proof}
Since $V_0$ is $\Z^n$-SI, the space $V_1:=D_A(V_0)$ is $A^{-1}\Z^n$-SI. Since
\bq\label{exp1}
V_1= V_0 + W_0, \qquad\text{where }W_0= \ov{\spa} E^{\Z^n}(\Psi),
\eq
the space $V_1$ is $\Z^n$-SI as well. Hence, $V_1$ is $(A^{-1}\Z^n+ \Z^n)$-SI. Consequently, $V_0$ is $(A\Z^n + \Z^n)$-SI.
\end{proof}

We are now ready to prove Theorem \ref{mb}.

\begin{proof}[Proof of Theorem \ref{mb}] 
By Theorem \ref{lr}, the shift-invariant projection $P_0$, given by \eqref{lr3}, is an integral operator with $\omega$-localized kernel, where $\omega=(1+\rho_A)^\eta$ and $\eta>1$. By Theorem \ref{ipt}, $V_0$ has a Riesz basis $E^{\Z^n}(\vp^1,\ldots,\vp^N)$ with each $\vp^j \in L^2(\omega)$. By the assumption that the space $V_0$ is $\Z^n$-SI and Lemma \ref{exp}, the space $V_0$ is also $A\Z^n$-SI. By Lemma \ref{irr}, $A\Z^n$ is a rational lattice. Consequently, $A$ is a rational matrix.
\end{proof}

\bibliographystyle{amsplain}

\end{document}